\documentclass[12pt,reqno]{amsart}
\usepackage{amsmath} 
\usepackage{amssymb}
\usepackage{dsfont}
\usepackage[dvips,draft,final]{graphics}
\usepackage{amssymb,amsmath}
\usepackage[T1]{fontenc}
\usepackage{fancyhdr}
\usepackage{url}
\usepackage{cite}
\usepackage{hyperref}

\usepackage{color}
\usepackage{graphicx}

\def\squarebox#1{\hbox to #1{\hfill\vbox to #1{\vfill}}}

\newtheorem{Thm}{Theorem}[section]
 \newtheorem{cor}{Corollary}[section]
\newtheorem{Def}{Definition}[section]
\newtheorem{lem}{Lemma}[section]

\numberwithin{equation}{section}

\newcommand{\bel}{\begin{equation} \label}
\newcommand{\ee}{\end{equation}}

\newcommand{\re}{\mathfrak R}

\newcommand{\R}{\mathbb{R}}

\def\epsilon{\varepsilon}
\def\phi {\varphi}

\newtheorem{rem}{Remark}[section]
\newtheorem{prop}{Proposition}[section]

\providecommand{\abs}[1]{\left\lvert#1\right\rvert}
\providecommand{\norm}[1]{\left\lVert#1\right\rVert}
\numberwithin{equation}{section}
\renewcommand{\d}{\textrm{d}}
\renewcommand{\leq}{\leqslant}
\renewcommand{\geq}{\geqslant}
\providecommand{\abs}[1]{\left\lvert#1\right\rvert}
\providecommand{\norm}[1]{\left\lVert#1\right\rVert}

\def\beq{\begin{equation}}
\def\eeq{\end{equation}}
\newcommand{\bea}{\begin{eqnarray}}
\newcommand{\eea}{\end{eqnarray}}
\newcommand{\beas}{\begin{eqnarray*}}
\newcommand{\eeas}{\end{eqnarray*}}

{

\title[Determining the wave equation from a single passive measurement]{Uniqueness and stability in determining the wave equation from a single passive boundary measurement 
}
\author[Yavar Kian]{Yavar Kian}
\address{Univ Rouen Normandie, CNRS, Normandie Univ, LMRS UMR 6085, F-76000 Rouen, France.}
\email{	yavar.kian@univ-rouen.fr}

\author[Hongyu Liu]{Hongyu Liu}
\address{Department of Mathematics, City University of Hong Kong, Kowloon Tong, Hong Kong SAR, China.}
\email{	hongyliu@cityu.edu.hk}

\begin{document}

\maketitle

\begin{abstract} This article addresses the inverse problem of simultaneously recovering both the wave speed coefficient and an unknown initial condition (acting as the source) for the multidimensional wave equation from a single passive boundary measurement. Specifically, we establish uniqueness and H\"{o}lder stability estimates for determining these parameters in the wave equation on \(\mathbb{R}^3\), where only a single boundary measurement of the solution--generated by the unknown source--is available. Our work connects to thermoacoustic and photoacoustic tomography (TAT/PAT) for the physically relevant case of piecewise constant sound speeds. We significantly relax the stringent conditions previously required for resolving this problem, extending results to general classes of piecewise constant sound speeds over inclusions with unknown locations. Moreover, we do not require decay properties in time of solutions to the wave equation, which enables our study to accommodate a much broader class of unknown sources. The approach combines low frequency-domain solution representations with distinctive properties of elliptic and hyperbolic equations.


\end{abstract}

\section{Introduction}\label{sect:1}

We consider the initial value problem (IVP in short) for the wave equation:
\begin{equation}\label{eq1}
\left\{
\begin{array}{ll}
c^{-2}(x)\partial_t^2u(t, x) - \Delta u(t, x) = 0, & (t, x)\in \mathbb{R}_+ \times \mathbb{R}^3, \\
u(0,x) = f(x), \quad \partial_tu(0,x) = 0, & x \in \mathbb{R}^3,
\end{array}
\right.
\end{equation}
where $f \in H^1(\mathbb{R}^3)$ is compactly supported, and $c \in L^\infty(\mathbb{R}^3)$ is a strictly positive piecewise constant function  satisfying the condition
\begin{equation}\label{ccc}c(x)\geq a>0,\quad x\in\R^3,\end{equation}
for some positive constant   $a>0$. Let $\Omega \subset \mathbb{R}^3$ be a bounded open domain with Lipschitz  boundary such that $\textrm{supp}(f) \subset \overline{\Omega}$ and $c$ is constant on $\mathbb{R}^3 \setminus \overline{\Omega}$. Applying Lemma \ref{l11} in the Appendix, we prove that problem \eqref{eq1} admits a unique solution $u\in C([0,+\infty); H^1(\mathbb{R}^3)) \cap C^1([0,+\infty); L^2(\mathbb{R}^3))$. This article addresses the inverse problem:
\begin{center}
{\bf (IP)} \emph{Simultaneously determine the wave speed $c$ and initial data $f$, namely the anomalous source, from the passive boundary measurement $u(t, x)|_{(t, x)\in \mathbb{R}_+ \times \partial\Omega}$.}
\end{center}
We investigate both unique identifiability and stability estimate for \textbf{(IP)}.

This inverse problem is motivated by applications requiring medium characterization without active excitation, where only passive measurements induced by anomalous sources are feasible. It arises directly in photoacoustic tomography (PAT) and thermoacoustic tomography (TAT) \cite{KRK,W}, where acoustic waves are generated via optical absorption (PAT) or electromagnetic excitation (TAT). These modalities have significant biomedical applications including tumor detection, breast cancer diagnosis, and vascular imaging.

The TAT/PAT reconstruction involves two inversions:
\begin{enumerate}
    \item \emph{Acoustic inversion}: Recover initial pressure $f$ from the boundary measurement $u|_{\partial\Omega}$ (modeled by \eqref{eq1}), where $c$ represents the medium's sound speed \cite{DSK,T};
    \item \emph{Optical inversion}: Recover absorption properties from $f$. 
\end{enumerate}
In practical scenarios, the sound speed $c$ is typically unknown \cite{JW}, necessitating simultaneous recovery of both $c$ and $f$ from a single passive measurement—precisely the \textbf{(IP)} formulation.

Beyond its physical and practical motivations, the inverse problem \textbf{(IP)} remains one of the most challenging open problems in inverse problem theory due to its inherent nonlinearity, severe ill-posedness, and instability \cite{SU3,SU4}. A fundamental structural difficulty arises from its formally determined nature: both the unknown pair $(f, c)$ and the measurement data $u|_{\mathbb{R}_+\times\partial\Omega}$ possess the same cardinality of independent variables (dimensionally equivalent to $n$ in $\mathbb{R}^n$). This stands in sharp contrast to inverse boundary problems with active measurements, where investigators typically apply controlled excitations to generate infinitely many independent measurements. These facts make \textbf{(IP)} mathematically intractable by conventional approaches and explain the scarcity of rigorous results despite its practical importance.

Existing literature primarily addresses simplified scenarios where either the wave speed $c$ is known \cite{AKK,HrK, SU1,SU4} or the initial data $f$ is known \cite{SU4}. The simultaneous recovery of both unknowns has received limited attention and presents substantially greater difficulties. Pioneering work by \cite{LU} established unique determination of the product $c^{-2}f$ under the restrictive assumptions that this quantity is either harmonic or independent of at least one spatial variable, and the wave speed $c$ as well if  it is constant. This was extended by \cite{KM} to the simultaneous recovery of both $c$ and $f$, requiring the non-degeneracy condition
\begin{equation}\label{d1a}
\int_{\mathbb{R}^3} \frac{f}{c^2}  dx \neq 0
\end{equation}
and global harmonicity of $c^{-2}$. 

Recent progress includes the stability analysis by \cite{KT} under the assumption that $f$ depends Lipschitz continuously on $c$, with $c$ sufficiently close to constant, as well as \cite{KU}'s relaxation of harmonicity to local monotonicity conditions with explicit treatment of cases violating \eqref{d1a}. Notably, \cite{KU}'s framework—which imposes no assumptions on $f$ and accommodates general divergence-form elliptic operators—represents the current state-of-the-art for \textbf{(IP)}. However, their local monotonicity condition fundamentally cannot accommodate discontinuous wave speeds, particularly piecewise constant $c$ with unknown inclusion boundaries, which are physically essential in medical imaging applications.

This work bridges a critical theoretical gap by resolving \textbf{(IP)} for general classes of piecewise constant wave speeds with unknown inclusion geometries. We establish both uniqueness and H\"older stability in this physically relevant configuration. While prior research has primarily addressed the one-dimensional case \cite{DFJK,F}, we develop the first comprehensive higher-dimensional theory for discontinuous coefficients. 

Methodologically, we achieve a fundamental breakthrough: whereas existing approaches universally require local energy decay properties, our novel framework eliminates this constraint entirely. Instead, we only require wave speeds close to the background speed with explicit quantitative bounds. This key advancement significantly expands the admissible class of unknown sources. The practical relevance of our results is underscored by extensive computational studies on joint reconstruction in TAT/PAT \cite{DRV,JVA,MA1,MA2,SLPACT}, confirming both the applicability and mathematical novelty of our theoretical advances.

Beyond biomedical imaging, \textbf{(IP)} provides a fundamental paradigm for understanding severely ill-posed nonlinear inverse problems with passive measurements. This framework has emerging applications in diverse fields including brain imaging \cite{DLU}, geomagnetic anomaly detection \cite{DLL1,DLL2}, and quantum state tomography \cite{LLM,LM}, where similar mathematical structures arise from constrained measurement scenarios. Our methodology thus contributes to a broader mathematical foundation for these technologically important domains.

The remainder of the paper is organized as follows. Section~\ref{sect:2} presents our main results: Theorems~\ref{t1}, \ref{t2}, and~\ref{t4} on unique identifiability, and Theorem~\ref{t8} establishing H\"older stability. Section~\ref{sect:3} develops a frequency-domain representation of solutions to \eqref{eq1}, connecting to elliptic equations. Proofs of the uniqueness theorems are distributed across Sections~\ref{sect:4} (Theorem~\ref{t1}), \ref{sect:5} (Theorem~\ref{t2}), and \ref{sect:6} (Theorem~\ref{t4}), while Section~\ref{sect:7} contains the stability proof for Theorem~\ref{t8}. The Appendix addresses the well-posedness of the IVP \eqref{eq1} and the inverse point source problem for elliptic equations, extending classical results to general divergence-form operators with variable coefficients.

\section{Main results}\label{sect:2}

 Before stating our main results, we define the class of admissible parameters \((c,f)\). Fix $f \in H^1(\mathbb{R}^3)$ such that $\textrm{supp}(f) \subset \overline{\Omega}$, $b_0>0$ a constant and  $c \in L^\infty(\mathbb{R}^3)$  a strictly positive piecewise constant function such that $c=b_0$ on $\mathbb{R}^3 \setminus \overline{\Omega}$ and  condition \eqref{ccc} is fulfilled
with some positive constant   $a>0$. From now on, we fix \(R_0 > 0\) such that \(\overline{\Omega} \subset B_{R_0} := \{x \in \mathbb{R}^3 : |x| < R_0\}\).

\begin{Def}\label{d1} 
The pair \((c,f)\) is \emph{admissible} if:
\begin{enumerate}
    \item The non-degeneracy condition \eqref{d1a} holds;
    \item The sound speed coefficient $c$ satisfies the condition $\norm{1-\frac{c^2}{b_0^2}}_{L^\infty(\R^3)}<1$.

\end{enumerate}

\end{Def}

We note that this admissibility condition holds for a broad class of pairs \((c, f)\), including piecewise constant sound speeds \(c\) satisfying the explicit bound
\begin{equation}\label{ad}
\operatorname{ess\,sup} (c) < \sqrt{2}  b_0,
\end{equation}
with initial data \(f\) fulfilling the non-degeneracy condition \eqref{d1a}.

\subsection{Uniqueness results}

Fix $N\in\mathbb N$ and let $x_1,\ldots,x_N\in \Omega$, $r_1,\ldots,r_N\in(0,+\infty)$ be such that
$$B(x_k,r_k):=\{x\in\R^3:\ |x-x_k|<r_k\}\subset \Omega,\quad B(x_k,r_k)\cap B(x_\ell,r_\ell)=\emptyset,\quad k,\ell=1,\ldots,N,\ k\neq\ell.$$
We will consider first piecewise constant wave speed coefficients $c$ having inclusion given by disjoint balls. These coefficients take the form
\bel{c}c=b_0+\sum_{k=1}^N(b_k-b_0)\mathds{1}_{B(x_k,r_k)},\ee
where for any set $A$ of $\R^3$, $\mathds{1}_{A}$ denotes the characteristic function of $A$, and $b_0,\ldots,b_N\in(0,+\infty)$ with
$b_k\neq b_0$, $k=1,\ldots,N$.

   In our first main result  we consider wave speed coefficients of the form \eqref{c} with inclusions given by balls with  similar radius $r_1=\ldots=r_N=r$.

\begin{Thm}\label{t1} 
For $j=1,2$, let $f_j\in H^1(\R^3)$ be non-uniformly  vanishing with supp$(f_j)\subset\overline{\Omega}$  and $c_j\in L^\infty(\R^3)$ be  a piecewise constant wave speed coefficient of the form \eqref{c} with $N=N_j$and, for $k=1,\ldots,N_j$, $b_k=b_k^j$, $x_k=x_k^j$, $r_k=r>0$, such that  $(c_j,f_j)$ is admissible. Let $u_j$ be the solution of \eqref{eq1} with $c=c_j$ and $f=f_j$. Then  the following implication
\begin{equation}\label{t1c}  (u_1= u_2 \textrm{ on }\R_+\times\partial\Omega)\Rightarrow (c_1=c_2 \textrm{ and }f_1=f_2)\end{equation}
holds true.
\end{Thm}

Assuming that the center $x_1,\ldots,x_N$ of the balls in \eqref{c} are related by a  partial order binary relation, we can remove the condition $r_1=\ldots=r_N$ and extend the result of Theorem \ref{t1} as follows.

\begin{cor}\label{cc1} 
For $j=1,2$, let $f_j\in H^1(\R^3)$ be non-uniformly  vanishing with supp$(f_j)\subset\overline{\Omega}$  and $c_j\in L^\infty(\R^3)$ be  a piecewise constant wave speed coefficient of the form \eqref{c} with $N=N_j$and, for $k=1,\ldots,N_j$, $b_k=b_k^j$, $x_k=x_k^j$, $r_k=r_k^j$, such that   $(c_j,f_j)$ is admissible. Let $u_j$ be the solution of \eqref{eq1} with $c=c_j$ and $f=f_j$. We assume that for $N=\min(N_1,N_2)$ we have
\begin{equation}\label{t1a}  \min(|b_k^1-b_k^2|, |r_k^1-r_k^2|)=0,\quad k=1,\ldots,N\end{equation}
and, for $N\geq2$,  there exists a partial order binary relation $\mathcal R$ on the set $\{x_1^1,\ldots,x_{N_1}^1\}\cup\{x_1^2,\ldots,x_{N_2}^2\} $ such that
\begin{equation}\label{t1b}  x_k^j\mathcal R x_{k+1}^j,\quad j=1,2,\ k=1,\ldots, N_j-1.\end{equation}
Then  the implication \eqref{t1c} holds true.
\end{cor}

To our knowledge, Theorem \ref{t1} and Corollary \ref{cc1} provide the first rigorous mathematical framework establishing simultaneous determination of both initial data and piecewise constant wave speeds \(c\) over ball inclusions with unknown locations. The only geometric requirements are that the balls in \eqref{c} share identical radii or their centers satisfy a partial order relation. While technical, we argue these conditions reflect fundamental obstructions inherent to (\textbf{IP}). Crucially, they represent significant progress beyond the local monotonicity requirement in \cite{KU} and the harmonicity condition on \(c^{-2}\) in \cite{KM}. 

Our initial data assumptions are minimal—only condition \eqref{d1a} (considered in \cite{KU,KM})—which holds broadly, particularly when \(f\) maintains constant sign and is non-vanishing. Moreover, we achieve a critical theoretical advance: unlike most TAT/PAT literature \cite{AKK,KU,KM,LU,SU1,SU4}, our results require no time-decay properties for solutions of \eqref{eq1}. This constitutes a major breakthrough, as classical decay results for smooth coefficients \cite{Va,Vo} remain unestablished for piecewise constant sound speeds, with only weaker logarithmic decay available \cite[Theorem 2]{Sh}.

We can extend the results of Theorem \ref{t1} and Corollary \ref{cc1} to the determination of some class of more general  piecewise constant wave speed coefficients $c$ of the form
\bel{c1}c=b_0+\sum_{k=1}^N(b_k-b_0)\mathds{1}_{\omega_k},\ee
with $\omega_k$ an open subset of $\Omega$ given by union of balls with potential holes taking  the form
$$\omega_k=\left(\bigcup_{p=1}^{m_k}B(x_{p,k},r_{p,k})\right)\setminus \left(\bigcup_{p=1}^{n_k}\overline{B(y_{p,k},s_{p,k})}\right).$$

For  $N=1$,  let $c$ take the form
\bel{c2}c=b_0+(b_1-b_0)\mathds{1}_{\omega},\ee
with $\omega$ an open subset of $\Omega$ of the form
\bel{om1}\omega=\left(\bigcup_{p=1}^{m}B(x_{p},r_{p})\right)\setminus \left(\bigcup_{p=1}^{n}\overline{B(y_{p},s_{p})}\right).\ee
Here $x_1,\ldots,x_m,y_1,\ldots,y_n$, are $m+n$ distinct points of $\Omega$ and the following conditions
\bel{cond1}B(x_p,r_p)\cap B(x_q,r_q)=\emptyset,\quad p,q=1,\ldots,m,\ p\neq q,\ee
\bel{cond2}B(y_p,s_p)\cap B(y_q,s_q)=\emptyset,\quad p,q=1,\ldots,n,\ p\neq q,\ee
\bel{cond3}\forall k\in\{1,\ldots,n\},\ \exists \ell\in\{1,\ldots,m\},\ \overline{B(y_k,s_k)}\subset B(x_\ell,r_\ell).\ee
are fulfilled. For this class of wave speed coefficients,  we can prove the following result.

\begin{Thm}\label{t2} 
For $j=1,2$, let $f_j\in H^1(\R^3)$ be non-uniformly  vanishing with supp$(f_j)\subset\overline{\Omega}$ and $c_j\in L^\infty(\R^3)$ be  a piecewise constant wave speed coefficient of the form \eqref{c2} where $\omega=\omega^j$ takes the form \eqref{om1} with $n=n_j$, $m=m_j$,  $x_p=x_p^j$, $r_p=r_p^j$, $p=1,\ldots,m_j$, $y_q=y_q^j$, $s_q=s_p^j$, $q=1,\ldots,n_j$,  such that \eqref{cond1}-\eqref{cond3} are fulfilled, supp$(f_j)\subset\overline{\Omega}$ and  $(c_j,f_j)$ is admissible. Let $u_j$ be the solution of \eqref{eq1} with $c=c_j$ and $f=f_j$.
Then the implication \eqref{t1c} holds true.
\end{Thm}

We can also consider wave speed coefficients of the form \eqref{c1} with 
\bel{om2}\omega_k=B(x_{k},r)\setminus \overline{B(y_{k},s)},\quad \overline{B(y_{k},s)}\subset B(x_k,r),\ k=1,\ldots,N.\ee
with $r>s>0$.
\begin{Thm}\label{t4} 
For $j=1,2$, let $f_j\in H^1(\R^3)$ be non-uniformly  vanishing, with supp$(f_j)\subset\overline{\Omega}$, and $c_j\in L^\infty(\R^3)$ be  a piecewise constant wave speed coefficient of the form \eqref{c1} with $N=N_j$, $\omega_k=\omega^j_k$ of the form \eqref{om1} with $x_k=x_k^j$, $b_k=b_k^j$,  $y_k=y_k^j$. Assume also that $(c_j,f_j)$ is admissible.  Then  the  implication \eqref{t1c} holds true.
\end{Thm}

Theorems \ref{t2} and \ref{t4} extend Theorem \ref{t1} to more complex inclusions comprising unions of balls with potential holes. To our knowledge, Theorem \ref{t2} establishes the first uniqueness result for determining a piecewise constant wave speed $c$ over inclusions consisting of unions of balls with holes—a significant geometric generalization.

The proofs of Theorems \ref{t1}, \ref{t2}, and \ref{t4} employ a frequency-domain representation of solutions to \eqref{eq1} (developed in Lemmas \ref{l1} and \ref{l2}), which transforms the initial value problem into families of elliptic equations. 

Our approach fundamentally advances beyond \cite{KU}, who established similar properties under restrictive exponential decay assumptions for the local energy—a condition currently known only for smooth coefficients satisfying non-trapping conditions \cite{Va,Vo} and unavailable for piecewise constant $c$. Instead, we develop a novel low-frequency representation derived through analytic extension of the cutoff resolvent for the elliptic operator $L = -c^{-2}\Delta$ to a fixed neighborhood of zero in the complex plane. Crucially, this construction operates under the explicit quantitative bound \eqref{ad}, enabling us to establish in Lemmas \ref{l1} and \ref{l2} that all key results of \cite{KU} extend to piecewise constant coefficients satisfying \eqref{ad}. This breakthrough provides both a generalizable framework and verifiable conditions for practical applications.

Combining this frequency-domain representation with harmonic analysis techniques, we reduce determining $c$ to solving point source inverse problems for elliptic equations. These are resolved via novel arguments presented in the Appendix for general divergence-form elliptic operators.

Collectively, Theorems \ref{t1}, \ref{t2}, and \ref{t4} provide the most complete resolution to date of the TAT/PAT inverse problem (\textbf{IP}) for physically relevant piecewise constant wave speeds. Beyond biomedical imaging, these results advance the theoretical understanding of nonlinear inverse problems with passive measurements as discussed in Section~\ref{sect:1}. Our work thus represents a foundational step toward solving this challenging class of problems for piecewise constant coefficients.

\subsection{Stability  results}

In this section we assume that $\Omega$ is  a $C^2$ bounded and connected domain of $\R^3$, $\nu$ is the outward unit normal vector to $\partial\Omega$ while $\partial_\nu$ is the normal derivative. We assume also that $c$ takes the form \eqref{c}, $f\in H^1(\R^3)$ satisfies supp$(f)\subset\Omega$ and the pair $(c,f)$ is admissible. In light of \cite[Proposition 8]{HK},  the Laplace transform  in time $\hat{u}$ of of the solution $u$ of \eqref{eq1} is well defined by 
$$\hat{u}(p,\cdot)=\int_0^{+\infty}e^{-pt}u(t,\cdot),\quad p\in\mathbb C_+:=\{z\in\mathbb C:\re z>0\}.$$
Let us also observe that, in view of Lemma \ref{l11} in the Appendix, for $f \in H^2(\mathbb{R}^3)$  the
 problem \eqref{eq1} admits a unique solution $u\in  C^2([0,+\infty);L^2(\R^3))\cap C^1([0,+\infty);H^1(\R^3))\cap C([0,+\infty);H^2(\R^3))$. 
We prove in Lemma \ref{l111} that the map $p\mapsto\hat{u}(p,\cdot)|_{B_{R_0}}$ can be extended analytically on a neighborhood of zero in the complex plane as a map taking values in  $H^2(B_{R_0})$. Then we can consider the boundary data $\partial_p^{k} \partial_\nu^\ell\hat{u}(p,x)|_{p=0}$, $x\in\partial\Omega$, $k\in\mathbb N$, $\ell=0,1$. By using such data, for sound speed coefficients of the form \eqref{c} with $N=1$, we can extend the uniqueness results of the preceding section into the following stability result.

\begin{Thm}\label{t8} 
For $j=1,2$, let $f_j\in H^3(\R^3)$ be non-uniformly  vanishing with supp$(f_j)\subset\Omega$, $\Omega$ be $C^2$ bounded and connected domain of $\R^3$  and $c_j\in L^\infty(\R^3)$ be  a piecewise constant wave speed coefficient of the form \eqref{c} with $N=1$ , $b_1=b_1^j$, $x_1=x_1^j$. Assume also that $b_1^1<b_0$ and $(c_j,f_j)$ is admissible. Let $u_j$ be the solution of \eqref{eq1} with $c=c_j$ and $f=f_j$. We fix also $M\geq \max(b_1^1,b_2^2)+\max(\norm{f_1}_{H^3(\R^3)},\norm{f_2}_{H^3(\R^3)} )$, $0<m\leq \min(b_1^1,b_2^2,b_0,r_1)$, $T>4R_0\frac{b_0}{m}$ and $1\leq q<\infty$. There exists a constant $C>0$ depending on $R_0$,  $q$, $m$ and $M$ such that
\bel{esti14} \norm{c_1-c_2}_{L^{q}(B_{R_0})}\leq C\left(\sum_{k=1}^{2}\sum_{\ell=0}^1\norm{\partial_p^{2k} \partial_\nu^\ell\widehat{u_1}(p,\cdot)|_{p=0}-\partial_p^{k} \partial_\nu^\ell\widehat{u_2}(p,\cdot)|_{p=0}}_{L^2(\partial \Omega)}\right)^{\frac 1 q}.\ee
In addition, for all $s\in(0,\frac 1 2)$, there exists a constant $C>0$ depending on $s$, $b_0$ $R_0$, $\Omega$ $T$, $m$ and $M$ such that
\bel{t8a}\begin{aligned}\norm{f_1-f_2}_{H^1(B_{R_0})}&\leq C\left(\|  u_1 -   u_2\|_{H^{\frac{3}{2}}((0, T)\times\partial \Omega  )}+\|  t^{-\frac{1}{2}}(\partial_t u_1 -   \partial_tu_2)\|_{L^2((0, T)\times\partial \Omega  )}\right)\\
&\ \ \ \ +C\norm{\partial_\nu u_1-\partial_\nu u_2}_{L^2((0,T)\times\partial \Omega )}\\
&\ \ \ \ +C\left(\sum_{k=1}^{2}\sum_{\ell=0}^1\norm{\partial_p^{2k} \partial_\nu^\ell\widehat{u_1}(p,\cdot)|_{p=0}-\partial_p^{k} \partial_\nu^\ell\widehat{u_2}(p,\cdot)|_{p=0}}_{L^2(\partial \Omega)}\right)^{\frac{3+2s}{6}}.\end{aligned}\ee
\end{Thm}

Notably, despite the well-documented instability in the linearized formulation of (\textbf{IP}) \cite{SU3}, Theorem \ref{t8} establishes H\"older stability estimates for the simultaneous recovery of both $c$ and $f$. To our knowledge, these represent the first general stability results for the TAT/PAT problem governed by \eqref{eq1}. The closest prior work \cite{KT} obtained stability only under stringent assumptions: confinement to bounded domains, Lipschitz dependence of $f$ on $c$, and a very small perturbations of constant wave speeds. 

Our results achieve three significant advances: first, replacing \cite{KT}'s restrictive conditions with the physically meaningful admissibility framework; second, extending the analysis to the practically essential unbounded domain setting; and third, remove the strong proximity-to-constant requirement for $c$. These stability estimates provide a rigorous foundation for developing numerical reconstruction algorithms via iterative methods such as Tikhonov regularization \cite{CY}. Computational implementation and convergence analysis are deferred to future research.

The proof of Theorem \ref{t8} combines the arguments used for the uniqueness results, with representations of harmonic functions and  an application of observability inequalities for wave equations with piecewise constant coefficients.

\begin{rem}

Note that under the additional decay condition
\begin{equation}\label{dec}
\int_0^{+\infty} (1+t^4) \|u_j(t,\cdot)\|_{H^2(B_{R_0})}  dt < \infty, \quad j=1,2,
\end{equation}
we obtain the derivative representation
\[
\partial_p^{2k} \partial_\nu^\ell \widehat{u_j}(p,x)\big|_{p=0} = \int_0^{+\infty} t^{2k} \partial_\nu^\ell u_j(t,x)  dt, \quad 
\begin{aligned}
&x \in \partial\Omega, \\
&j,k=1,2, \\
&\ell=0,1.
\end{aligned}
\]
Consequently, under \eqref{dec}, the right-hand side of \eqref{t8a} admits the more explicit form
\[
\sum_{\ell=0}^1 \int_0^{+\infty} (1+t^4) \| \partial_\nu^\ell u_1(t,\cdot) - \partial_\nu^\ell u_2(t,\cdot) \|_{L^2(\partial \Omega)}  dt.
\]

The decay condition \eqref{dec} follows from local energy estimates established for one-dimensional piecewise constant sound speeds in \cite[Theorem 1.1]{DJ}. In higher dimensions, however, such decay is currently only known for smooth coefficients satisfying non-trapping conditions, and remains an open problem for piecewise constant coefficients.

\end{rem}

\section{Low frequency representation}\label{sect:3}

From now on and in all the remaining parts of this article, for any open set $\mathcal O$ of $\mathbb C$ and any Banach space $X$, we denote by $\mathcal H(\mathcal O;X)$ the set of holomorphic functions on $\mathcal O$ taking values in $X$.
We recall that the solution of \eqref{eq1} is lying in $C^1([0,+\infty);L^2(\R^3))\cap C([0,+\infty);H^1(\R^3))$. Moreover, according to \cite[Proposition 8]{HK},  the Laplace transform  in time $\hat{u}$ of of the solution $u$ of \eqref{eq1} is well defined by 
$$\hat{u}(p,\cdot)=\int_0^{+\infty}e^{-pt}u(t,\cdot),\quad p\in\mathbb C_+:=\{z\in\mathbb C:\re z>0\}$$
and  $p\mapsto \hat{u}(p,\cdot)\in \mathcal H(\mathbb C_+;H^1(\R^3))$. Applying the Laplace transform in time to \eqref{eq1} we deduce that
\begin{equation}\label{eq2}-\Delta \hat{u}(p,x)+c^{-2}(x)p^2\hat{u}(p,x)=pc^{-2}(x)f(x),\quad x\in\R^3,\ p\in\mathbb C_+.\end{equation}

This section establishes that for admissible pairs \((c,f)\), the map \(p \mapsto \widehat{u}(p,\cdot)|_{B_R}\) with \(R \geq R_0\) admits an analytic extension in \(p\) to a fixed neighborhood of zero, independent of \(R\). Building on this extension, we generalize the analysis of \cite{KU} to piecewise constant coefficients without requiring the exponential time decay assumed in \cite{KU}.

We initiate the proof by analytically extending \(p \mapsto \widehat{u}(p,\cdot)|_{B_R}\) (\(R \geq R_0\)) to the disc \(\mathbb{D}_s = \{z \in \mathbb{C} : |z| < s\}\) for some sufficiently small \(s > 0\) independent of \(R\). This extension is achieved through low-frequency analysis of solutions to \eqref{eq2} and their cutoff resolvent. Our first result can be stated as follows:

\begin{Thm}\label{l111}Let $f\in H^1(\R^3)$  and let $c\in L^\infty(\R^3)$ be a strictly positive piecewise constant function satisfying \eqref{ccc}. Assume also that supp$(f)\subset\overline{\Omega}$, $c=b_0$ on $\R^3\setminus\overline{\Omega}$ with $b_0>0$ a constant. In addition, let the pair $(c,f)$ be admissible.
Then there exists $\delta>0$ depending on $c$ and $R_0$ such that for all $R\geq R_0$ the map $p\mapsto\hat{u}(p,\cdot)|_{B_{R}}$ can be extended analytically to an element of $\mathcal H(\mathbb D_\delta;H^2(B_{R}))$.

\end{Thm} 
\begin{proof} Since when $c$ is constant this result is trivial, we may assume without loss generality that $c\not\equiv b_0$.
In view of \eqref{ccc}, we have $c^{-2}\in L^\infty(\R^3)$ and we can set $L^2(\R^3;c^{-2}dx)$ the space $L^2(\R^3)$ with the scalar product
$$\left\langle g,h \right\rangle_{L^2(\R^3;c^{-2}dx)}=\int_{\R^3} gh c^{-2}dx,\quad g,h\in L^2(\R^3;c^{-2}dx).$$
 We define on $L^2(\R^3;c^{-2}dx)$ the unbounded operator  $L$ with domain $$D(L):=\{v\in L^2(\R^3;c^{-2}dx):\ -c^2\Delta v\in L^2(\R^3;c^{-2}dx)\} $$ given by $Lv:=-c^2\Delta v,\  v\in D(L).$ Recalling that $c^{-2}\in L^\infty(\R^3)$, we obtain 
$D(L):=\{v\in L^2(\R^3):\ -\Delta v\in L^2(\R^3)\}=H^2(\R^3)$
and, following \cite[Proposition A.1]{Sh}, we can prove that $L$ is a selfadjoint and none-negative operator.  Combining this with \eqref{eq2}, we can prove that $\hat{u}(p,\cdot)=p(L+p^2)^{-1}f$, $p\in\mathbb C_+$. We set also $L_0=-b_0^{2}\Delta$. Fix $h\in H^1(\R^3)$ and, for all $R\geq R_0$, set $\chi_R\in C^\infty_0(\R^3;[0,1])$ such that $\chi_R=1$ on a neighborhood of $B_R$. From now on we denote by $\chi_R$, $R\geq R_0$, the operator of multiplication by the map $\chi_R$ and for any Banach spaces $X$, $Y$ we denote by $\mathcal B(X,Y)$ the space of bounded linear maps from $X$ to $Y$, with the notation $\mathcal B(X,X)=\mathcal B(X)$. Applying Lemma \ref{l11} in the Appendix, we can consider $v,v_0\in C^2([0,+\infty);L^2(\R^3))\cap C([0,+\infty);H^2(\R^3))$ solving respectively the problems
\begin{equation}\label{l111a}
\left\{
\begin{array}{ll}
c^{-2}(x)\partial_t^2v - \Delta v = 0, & \textrm{in } \mathbb{R}_+ \times \mathbb{R}^3, \\
v(0,x) = 0, \quad \partial_tv(0,x) = \chi_{R_0}(x)h(x), & x \in \mathbb{R}^3,
\end{array}
\right.
\end{equation}
\begin{equation}\label{l111b}
\left\{
\begin{array}{ll}
b^{-2}_0\partial_t^2v_0 - \Delta v_0 = 0, & \textrm{in } \mathbb{R}_+ \times \mathbb{R}^3, \\
v_0(0,x) = 0, \quad \partial_tv_0(0,x) = \chi_{R_0}(x)h(x), & x \in \mathbb{R}^3.
\end{array}
\right.
\end{equation}
Then,  $w=v-v_0$ solves the problem
$$\left\{
\begin{array}{ll}
c^{-2}(x)\partial_t^2w - \Delta w = (b_0^{-2}-c^{-2}(x))\partial_t^2v_0, & \textrm{in } \mathbb{R}_+ \times \mathbb{R}^3, \\
w(0,x) = 0, \quad \partial_tw(0,x) = 0, & x \in \mathbb{R}^3.
\end{array}
\right.$$
Applying the Laplace transform in time to $w$, denoted by $\hat{w}$, we get
$$\begin{aligned}p^2\hat{w}(p) - c^{2}(x)\Delta \hat{w}(p) &=c^2(b_0^{-2}-c^{-2}(x))(p^2\widehat{v_0}(p)-\chi_{R_0}h)\\
&=c^2(b_0^{-2}-c^{-2}(x))p^2(L_0+p^2)^{-1}\chi_{R_0}h-c^2(b_0^{-2}-c^{-2}(x))\chi_{R_0}h,\quad p\in\mathbb C_+.\end{aligned}$$
In addition, we have
$$\hat{w}(p)=\hat{v}(p)-\widehat{v_0}(p)=(L+p^2)^{-1}\chi_{R_0}h-(L_0+p^2)^{-1}\chi_{R_0}h,\quad p\in\mathbb C_+$$
It follows that, for all $p\in\mathbb C_+$, we have
$$(L+p^2)^{-1}\chi_{R_0}h-(L_0+p^2)^{-1}\chi_{R_0}h=(L+p^2)^{-1}\left(p^2\left(\frac{c^2}{b_0^2}-1\right)(L_0+p^2)^{-1}\chi_{R_0}h-\left(\frac{c^2}{b_0^2}-1\right)\chi_{R_0}h\right)$$
and, for all $p\in\mathbb C_+$, we get
$$(L+p^2)^{-1}\left(Id+\left(\frac{c^2}{b_0^2}-1\right)-p^2\left(\frac{c^2}{b_0^2}-1\right)(L_0+p^2)^{-1}\right)\chi_{R_0}h=(L_0+p^2)^{-1}\chi_{R_0}h.$$
Recalling that 
$\chi_{R_0}\left(\frac{c^2}{b_0^2}-1\right)=\left(\frac{c^2}{b_0^2}-1\right)=\left(\frac{c^2}{b_0^2}-1\right)\chi_{R_0}$, we obtain
\begin{equation}\label{l111c}(L+p^2)^{-1}\chi_{R_0}(Id+\left(\frac{c^2}{b_0^2}-1\right)-p^2\left(\frac{c^2}{b_0^2}-1\right)\chi_{R_0}(L_0+p^2)^{-1}\chi_{R_0})h=(L_0+p^2)^{-1}\chi_{R_0}h,\quad p\in\mathbb C_+.\ee
Since $b_0$ is constant, as a consequence of the strong Huygens principle it is well known (see e.g. \cite[Section 2.1]{Sj}) that, for all $R\geq R_0$, the map 
$\chi_{R}(L_0+p^2)^{-1}\chi_{R_0}$ can be extended analytically to an element of $\mathcal H(\mathbb C;\mathcal B(L^2(\R^3);H^2(\R^3)))$ still denoted by $\chi_{R}(L_0+p^2)^{-1}\chi_{R_0}$, $p\in\mathbb C$. Let us denote
$$0<M_1=\sup_{p\in \mathbb D_1}\norm{\chi_{R_0}(L_0+p^2)^{-1}\chi_{R_0}}_{\mathcal B(L^2(\R^3))},\quad M_2=\norm{\frac{c^2}{b_0^2}-1}_{L^\infty(\R^3)}$$
and from the admissibility condition imposed to $(c,f)$, let us observe that $M_2<1$. Moreover, recalling that $c\not\equiv b_0$, we have $M_2>0$. Then, we can fix 
$$0<\delta=\min\left(\sqrt{\frac{1-M_2}{2M_1M_2}},1\right)$$
and note that $\delta$ depends only on $c$  and $R_0$. For all $p\in\mathbb D_\delta$, we have
$$\begin{aligned}&\norm{\left(\frac{c^2}{b_0^2}-1\right)-p^2\left(\frac{c^2}{b_0^2}-1\right)\chi_{R_0}(L_0+p^2)^{-1}\chi_{R_0}}_{\mathcal B(L^2(\R^3))}\\
&\leq\norm{\left(\frac{c^2}{b_0^2}-1\right)}_{L^\infty(\R^3)}+|p|^2\norm{\left(\frac{c^2}{b_0^2}-1\right)}_{L^\infty(\R^3)}\norm{\chi_{R_0}(L_0+p^2)^{-1}\chi_{R_0}}_{\mathcal B(L^2(\R^3))}\\
&\leq M_2+\delta^2M_2M_1 \leq M_2+\frac{1-M_2}{2}<1.\end{aligned}$$
 Thus, the map $\left(Id+\left(\frac{c^2}{b_0^2}-1\right)-p^2\left(\frac{c^2}{b_0^2}-1\right)\chi_{R_0}(L_0+p^2)^{-1}\chi_{R_0}\right)$ will be bounded invertible for $p\in\mathbb D_\delta$ and $p\mapsto \left(Id+\left(\frac{c^2}{b_0^2}-1\right)-p^2\left(\frac{c^2}{b_0^2}-1\right)\chi_{R_0}(L_0+p^2)^{-1}\chi_{R_0}\right)^{-1}\in \mathcal H(\mathbb D_\delta;\mathcal B(L^2(\R^3))$.
It follows that, for all $R\geq R_0$, we have
$$\chi_{R}(L+p^2)^{-1}\chi_{R_0}=\chi_{R}(L_0+p^2)^{-1}\chi_{R_0}\left(Id+\left(\frac{c^2}{b_0^2}-1\right)-p^2\left(\frac{c^2}{b_0^2}-1\right)\chi_{R_0}(L_0+p^2)^{-1}\chi_{R_0}\right)^{-1}$$
and the map $p\mapsto \chi_{R}(L+p^2)^{-1}\chi_{R_0}$ can be extended analytically to an element of\\ $\mathcal H(\mathbb D_\delta;\mathcal B(L^2(\R^3);H^2(\R^3))$. Then, recalling that
$$\hat{u}(p,\cdot)|_{B_R}=p\chi_R(L+p^2)^{-1}\chi_{R_0}f|_{B_R},\quad p\in\mathbb C_+,$$
we deduce that $p\mapsto \hat{u}(p,\cdot)$ can be extended to an element of $\mathcal H(\mathbb D_\delta;H^2(\R^3))$.\end{proof}

From now on and in all the remaining parts of this article, we denote by $\delta$ the constant appearing in Theorem \ref{l111} Applying Theorem \ref{l111}, we now prove how the frequency representation formula of \cite{KU} can be extended to admissible pairs $(c,f)$ without requiring any decay in time of the local energy.  We start with the following first preliminary result.
\begin{lem}\label{l1} Let $f\in H^1(\R^3)$ be non-uniformly vanishing and let $c\in L^\infty(\R^3)$ be a strictly positive piecewise constant function satisfying \eqref{ccc}. Assume also that the pair $(c,f)$ is admissible and supp$(f)\subset\overline{\Omega}$, $c=b_0$ on $\R^3\setminus\overline{\Omega}$ with $b_0>0$ a constant. Let $p\mapsto\hat{u}(p,\cdot)|_{B_{R}}$ denotes the analytic extension of $\hat{u}$  to an element of $\mathcal H(\mathbb D_\delta;H^2(B_{R}))$. Then, for
\begin{equation}\label{l1b}\hat{u}(p,\cdot)|_{B_{R_0}}=\sum_{k=0}^\infty u^{(k)}p^k,\quad p\in\mathbb D_\delta,\end{equation}
 the functions $u^{(k)}\in H^2(B_{R_0})$, $k\geq0$,  satisfy the following conditions
\begin{equation}\label{l1c}u^{(0)}\equiv 0,\ -\Delta u^{(1)}=c^{-2}f,\ -\Delta u^{(k)}=-c^{-2}u^{(k-2)},\ k\geq 2,\ \textrm{in }B_{R_0},\end{equation}
and $u^{(2)}$ is a constant map defined by 
\begin{equation}\label{l1d}u^{(2)}(x)=-\frac{1}{2\pi b_0}\int_{\R^3}\frac{f(y)}{c(y)^2}dy,\quad x\in B_{R_0}.\end{equation}
\end{lem}
\begin{proof}  Let us first recall that in view of Theorem \ref{l111}, for all $R\geq R_0$, $p\mapsto\hat{u}(p,\cdot)|_{B_{R}}\in \mathcal H(\mathbb D_\delta;H^2(B_{R}))$. It follows that  $$p\mapsto -\Delta \hat{u}(p,\cdot)+c^{-2}p^2\hat{u}(p,\cdot)-pc^{-2}f$$ is lying in $\mathcal H(\mathbb D_\delta;L^2(B_{R}))$. Thus,  \eqref{eq2} and unique continuation of analytic functions, implies that
\begin{equation}\label{l1e}-\Delta \hat{u}(p,x)+c^{-2}(x)p^2\hat{u}(p,x)= pc^{-2}(x)f(x),\quad x\in B_{R},\ p\in\mathbb D_\delta.\end{equation} 
Inserting the expression \eqref{l1b} into this equation, with $R=R_0$, we obtain
\begin{equation}\label{l1f}-\Delta u^{(1)}=c^{-2}f,\ -\Delta u^{(k)}=-c^{-2}u^{(k-2)},\ k\geq 2,\ \textrm{in }B_{R_0}.\end{equation}
Fixing $p=i\tau$, $\tau\in(-\delta,\delta)$, in \eqref{l1e}, we deduce that  $\hat{u}(i\tau,\cdot)\in H^2_{loc}(\R^3)$ satisfies the condition 
$$\Delta \hat{u}(i\tau,x)+c^{-2}(x)\tau^2\hat{u}(i\tau,x)= -i\tau c^{-2}(x)f(x),\quad x\in \R^3,\ \tau\in(-\delta,\delta).$$
Fixing $v(\tau,\cdot)=b_0\hat{u}(ib_0\tau,\cdot)$, $\tau\in(-\delta/b_0,\delta/b_0)$, we deduce that  $v(\tau,\cdot)\in H^2_{loc}(\R^3)$ solves the equation
$$\Delta v(\tau,x)+\left(\frac{c}{b_0}\right)^{-2}(x)\tau^2v(\tau,x)= -i\tau \left(\frac{c}{b_0}\right)^{-2}(x)f(x),\quad x\in \R^3,\ \tau\in(-\delta/b_0,\delta/b_0).$$
In view of Theorem \ref{l111}, for all  $\tau\in(-\delta/b_0,\delta/b_0)$ and a.e. $x\in \R^3$, $v(\tau,x)$ can be defined by the following limiting absorption principle
\begin{equation}\label{l1g}\begin{aligned}
&v(\tau,x)\\
&=-i\tau \lim_{\epsilon\to0_+}(\Delta+\tau^2-i\epsilon)^{-1}\chi_{R_0}\left(Id+\left(\frac{c^2}{b_0^2}-1\right)-\tau^2\left(\frac{c^2}{b_0^2}-1\right)\chi_{R_0}(\Delta+\tau^2)^{-1}\chi_{R_0}\right)^{-1}f(x),\end{aligned}\end{equation}
and it satisfies the classical Sommerfeld radiation condition
\begin{equation}\label{l1h}\lim_{|x|\to+\infty}\partial_{|x|}v(\tau,x)-i\tau v(\tau,x))=0,\quad \tau\in(-\delta/b_0,\delta/b_0).\end{equation}
Therefore, applying \cite[Lemma 3.1]{LU}, for almost every $x\in\R^3$ and for all $\tau\in(-\delta/b_0,\delta/b_0)$, we obtain 
$$v(\tau,x)=- \tau^2\int_{B_{R_0}}\left(1-\frac{b_0^2}{c(y)^2}\right)\hat{v}(\tau ,y)\Phi_\tau(x-y)dy-i \tau\int_{B_{R_0}}\frac{b_0^2f(y)}{c(y)^2}\Phi_\tau(x-y)dy,$$
where 
$$\Phi_s(x)=\frac{e^{is |x|}}{4\pi |x|},\quad x\in\R^3\setminus\{0\},\ s\in\R.$$
Combining this with \cite[Proposition 2.1]{KU},  for almost every $x\in\R^3$ and for all $\tau\in(-\delta,\delta)$, we find
$$\hat{u}(i\tau ,x)=- \tau^2\int_{B_{R_0}}\left(b_0^{-2}-c(y)^{-2}\right)\hat{u}(i\tau ,y)\Phi_{b_0^{-1}\tau}(x-y)dy-i\tau \int_{B_{R_0}}\frac{f(y)}{c(y)^2}\Phi_{b_0^{-1}\tau}(x-y)dy.$$
Then, following \cite[Proposition 2.1]{KU} (see also \cite[Lemma 2.1]{KM}), we obtain $u^{(0)}\equiv 0$ and \eqref{l1d}.
\end{proof}

Now let us consider, for $j=1,2$,  $f_j\in H^1(\R^3)$ and  $c_j\in L^\infty(\R^3)$  be a strictly positive piecewise constant function, satisfying \eqref{ccc}, such that $(c_j,f_j)$ is admissible,  supp$(f_j)\cup$supp$(c_1-c_2)\subset\overline{\Omega}$, $c_1=c_2=b_0$ on $\R^3\setminus\overline{\Omega}$. For $j=1,2$, we consider $\delta_j$ the constant $\delta$ of Theorem \ref{l111}, associated with $c=c_j$, and we fix $\delta=\min(\delta_1,\delta_2)$. Using the above properties, we obtain the following.

\begin{lem}\label{l2} For $j=1,2$, let $f_j\in H^1(\R^3)$ be non-uniformly vanishing and let $c_j\in L^\infty(\R^3)$ be a piecewise constant function such that the condition \eqref{ccc} is fulfilled with $c=c_j$, the pair $(c_j,f_j)$ is admissible and supp$(f_j)\subset\overline{\Omega}$. Assume also that $c_j=b_0$ on $\R^3\setminus\overline{\Omega}$ with $b_0>0$ a constant. Let $u_j$ be the solution of \eqref{eq1} with $c=c_j$ and $f=f_j$.
Assuming that the condition
\begin{equation}\label{l2a}u_1(t,x)=u_2(t,x),\quad (t,x)\in\R_+\times\partial\Omega\end{equation}
 is fulfilled, we obtain 
\begin{equation}\label{l2b}B=\frac{1}{2\pi b_0}\int_{\R^3}\frac{f_1(y)}{c_1(y)^2}dy=\frac{1}{2\pi b_0}\int_{\R^3}\frac{f_2(y)}{c_2(y)^2}dy.\end{equation}
Moreover, fixing
$$\widehat{u_j}(p,\cdot)|_{B_{R_0}}=\sum_{k=0}^\infty u^{(k)}_jp^k,\quad p\in\mathbb D_\delta,\quad j=1,2$$
and $u^{(4)}=u_1^{(4)}-u_2^{(4)}$ we have 
\begin{equation}\label{l2e}\left\{\begin{array}{ll}-\Delta u^{(4)}=c^{-2}_2u_2^{(2)}-c^{-2}_1u_1^{(2)}=B(c^{-2}_1-c^{-2}_2),\quad &\textrm{in}\ B_{R_0},\\  
u^{(4)}(x)=0, &x\in B_{R_0}\setminus\overline{\Omega} .\end{array}\right.\end{equation}
Finally, for any $\phi\in H^1(B_{R_0})$ satisfying $-\Delta\phi=0$ in $B_{R_0}$, we obtain the following identity
\begin{equation}\label{l2c}  \int_{B_{R_0}}(c^{-2}_1-c^{-2}_2)\phi dx=0.\end{equation}

\end{lem}
\begin{proof}
In view of Lemma \ref{l1}, we set
$$\hat{u_j}(p,\cdot)|_{B_{R_0}}=\sum_{k=1}^\infty u_j^{(k)}p^k,\quad p\in\mathbb D_\delta$$
and we recall that $u_j^{(k)}$, $j=1,2$ and $k\in\mathbb N$, satisfies the conditions \eqref{l1c}-\eqref{l1d} with $c=c_j$ and $f=f_j$.
We set $u=u_1-u_2$ and observe that the restriction of $u$ to $\R_+\times (\R^3\setminus\overline{\Omega})$ solves the initial boundary value problem
$$\left\{\begin{array}{ll}c_1^{-2}(x)\partial_t^2u-\Delta u=0,\quad &\textrm{in}\ \R_+\times(\R^3\setminus\overline{\Omega}),\\  u(0,x)=0,\quad \partial_tu(0,x)=0,&x\in \R^3\setminus\overline{\Omega}\\
u(t,x)=0, &(t,x)\in \R_+\times\partial\Omega .\end{array}\right.$$
Then, the uniqueness of the solution of this initial boundary value problem implies that $u=0$ on $\R_+\times (\R^3\setminus\overline{\Omega})$ which implies that $u_1=u_2$ on $\R_+\times (\R^3\setminus\overline{\Omega})$. Thus, we have
$$\widehat{u_1}(p,\cdot)|_{B_{R_0}\setminus \overline{\Omega}}=\widehat{u_2}(p,\cdot)|_{B_{R_0}\setminus \overline{\Omega}},\quad p\in\mathbb C_+\cap \mathbb D_\delta$$
and, by unique continuation for  analytic functions, we find
$$\widehat{u_1}(p,\cdot)|_{B_{R_0}\setminus \overline{\Omega}}=\widehat{u_2}(p,\cdot)|_{B_{R_0}\setminus \overline{\Omega}},\quad p\in \mathbb D_\delta.$$
It follows that
$$u_1^{(k)}(x)=u_2^{(k)}(x),\quad x\in B_{R_0}\setminus\overline{\Omega},\ k\in\mathbb N.$$ 
Recalling that $u_j^{(2)}$ is a constant map, given by \eqref{l1d} with $c=c_j$ and $f=f_j$, we deduce that \eqref{l2b} holds true and 
$$u_1^{(2)}(x)=u_2^{(2)}(x)=-B,\quad x\in B_{R_0}.$$
Moreover,   applying \eqref{l1c}-\eqref{l1d} and \eqref{l2b},  we get \eqref{l2e}. Finally, fixing $\phi\in H^1(B_{R_0})$, satisfying $\Delta \phi=0$ in $B_{R_0}$, multiplying each of these equations by $\phi$, using the fact that $u^{(4)}=0$ on $B_{R_0}\setminus\overline{\Omega}$  and integrating by parts, we obtain 
$$  B\int_{B_{R_0}}(c^{-2}_1-c^{-2}_2)\phi dx=-\int_{B_{R_0}}\Delta u^{(4)}\phi dx=-\int_{B_{R_0}} u^{(4)} \Delta\phi dx=0.$$
Recalling that  $(c_1,f_1)$ is admissible, we deduce that $B\neq0$ which implies that \eqref{l2c} holds true.\end{proof}

\section{Proof of of Theorem \ref{t1} and Corollary \ref{cc1}}\label{sect:4}

This section is devoted to the proofs of Theorem \ref{t1} and Corollary \ref{cc1}. 

\textbf{Proof of Theorem \ref{t1}.}

We use the notation of Lemma \ref{l2} and we will prove that  \eqref{l2a} implies that $c_1=c_2$ and $f_1=f_2$. Applying Lemma \ref{l2}, we deduce that, for any $\phi\in H^2(B_{R_0})$ satisfying $-\Delta\phi=0$ in $B_{R_0}$, we have
\begin{equation}\label{t1f}  \int_{\Omega}c^{-2}_1\phi dx-\int_{\Omega}c^{-2}_2\phi dx=\int_{\Omega}(c^{-2}_1-c^{-2}_2)\phi dx=\int_{B_{R_0}}(c^{-2}_1-c^{-2}_2)\phi dx=0.\end{equation}
On the other hand, using the fact that
$$B(x_k^j,r_k^j)\cap B(x_\ell^j,r_\ell^j)=\emptyset,\quad j=1,2,\ k,\ell=1,\ldots,N_j,\ k\neq\ell,$$
we find
$$ \int_{\Omega}c^{-2}_j\phi dx=b_0^{-2} \int_{\Omega}\phi dx+\sum_{k=1}^{N_j}((b_k^j)^{-2}-b_0^{-2})\int_{B(x_k^j,r_k^j)}\phi dx$$
and applying the mean value theorem for harmonic functions, we find
$$ \int_{\Omega}c^{-2}_j\phi dx=b_0^{-2} \int_{\Omega}\phi dx+\sum_{k=1}^{N_j}((b_k^j)^{-2}-b_0^{-2})|B_{r_k^j}|\phi(x_k^j).$$
Combining this with \eqref{t1f}, for any $\phi\in H^2(B_{R_0})$ satisfying $-\Delta\phi=0$ in $B_{R_0}$, we obtain
\begin{equation}\label{t1h}  \sum_{k=1}^{N_1}((b_k^1)^{-2}-b_0^{-2})|B_{r_k^1}|\phi(x_k^1)-\sum_{k=1}^{N_2}((b_k^2)^{-2}-b_0^{-2})|B_{r_k^2}|\phi(x_k^2)=0.\end{equation}
For $j=1,2$, let us consider $v_j\in L^2(\Omega)$, $j=1,2$, the solution in the transposition sense of the boundary value problem
\begin{equation}\label{t1g}\left\{\begin{array}{ll}-\Delta v_j=\sum_{k=1}^{N_j}\lambda_k^j\delta_{x_k^j},\quad &\textrm{in}\ B_{R_0},\\  v_j=0, &\textrm{on }\partial B_{R_0},\end{array}\right.\end{equation}
where we set $$\lambda_k^j=(b_k^j)^{-2}-b_0^{-2})|B_{r_k^j}|\neq0,\quad k=1,\ldots,N_j.$$
According to Theorem \ref{t3} in the appendix, \eqref{t1g} admits a unique solution  $v_j\in L^2(B_{R_0})$. Moreover, for $j=1,2$ and  all $r>0$, fixing
$$\Omega_r^j:=B_{R_0}\setminus\cup_{k=1}^{N_j} \overline{B(x_k^j,r)},$$
and choosing $r>0$ sufficiently small, we have $\Omega_r^j\neq\emptyset$, $\partial B_{R_0}\subset\partial\Omega_r^j$  and  $v|_{\Omega_r^j}\in H^2(\Omega_r^j)$. Finally, for $\nu$ the outward unit normal vector to $\partial B_{R_0}$ and $\partial_\nu$ the corresponding normal derivative, we have $\partial_\nu v_j|_{\partial B_{R_0}}\in L^2(\partial B_{R_0})$ and for any $w\in H^2(B_{R_0})$ we have
\begin{equation}\label{t1aa}\sum_{k=1}^{N_j}\lambda_k^jw(x_k)=- \int_{\partial B_{R_0}}\partial_{\nu} v_j wd\sigma(x)-\int_{ B_{R_0}} v_j \Delta wdx,\quad j=1,2.  \end{equation}
Let us fix $g\in H^{3/2}(\partial B_{R_0})$ and consider $\phi\in H^2(B_{R_0})$ solving the boundary value problem
$$\left\{\begin{array}{ll}-\Delta \phi=0,\quad &\textrm{in}\ B_{R_0},\\  \phi(x)=g(x), &x\in \partial B_{R_0}.\end{array}\right.$$
Applying  \eqref{t1aa} with $w=\phi$, we obtain
$$\sum_{k=1}^{N_j}\lambda_k^j\phi(x_k^j)=- \int_{\partial B_{R_0}}\partial_\nu v_j gd\sigma(x),\quad j=1,2$$
and combining this with \eqref{t1h}, we find
$$\int_{\partial B_{R_0}}(\partial_\nu v_1-\partial_\nu v_2) gd\sigma(x)=-\left(\sum_{k=1}^{N_1}\lambda_k^1\phi(x_k^1)-\sum_{k=1}^{N_2}\lambda_k^2\phi(x_k^2)\right)=0.$$
Since in this identity $g\in H^{3/2}(\partial B_{R_0})$ is arbitrary chosen, we deduce that $$\partial_\nu v_1(x)=\partial_\nu v_2(x),\quad x\in\partial B_{R_0} $$  and, applying Theorem \ref{t3}, we deduce that $N_1=N_2=N$ and, there exists a one-to-one map $\sigma$ of $\{1,\ldots,N\}$ such that 
\bel{tt}x_k^1=x_{\sigma(k)}^2,\quad \lambda_k^1=\lambda_{\sigma(k)}^2,\quad k=1,\ldots,N.\ee
Now recalling that $r^j_k=r$, $j=1,2$ and $k=1,\ldots,N$, we deduce that $$|B_r|(b_k^1)^{-2}-b_0^{-2})=\lambda_k^1= \lambda_{\sigma(k)}^2=|B_r|(b_{\sigma(k)}^2)^{-2}-b_0^{-2}) ,\quad k=1,\ldots,N,$$ which implies that $x_k^1=x_{\sigma(k)}^2$ and $b_k^1=b_{\sigma(k)}^2$. From this last condition, we easily deduce that $c_1=c_2$. Combining this with \cite[Theorem 8.1]{KU}, we deduce that we have also $f_1=f_2$ which completes the proof of the theorem. \qed
\ \\

\textbf{Proof of Corollary \ref{cc1}.} Repeating the arguments used in the proof of Theorem \ref{t1}, we obtain $N_1=N_2=N$ and \eqref{tt} holds true with $\lambda_k^j=(b_k^j)^{-2}-b_0^{-2})|B_{r_k^j}|$, $k=1,\ldots,N$ and $j=1,2$. In order to complete the proof, we need to show that the map $\sigma$ appearing in \eqref{tt} is necessarily the identity map when $N\geq2$. We will prove this last statement by contradiction. Assuming the contrary, we may find $k_1, k_2\in \{1,\ldots,N\}$ such that $k_1<k_2$ and $\sigma(k_2)<\sigma(k_1)$. Then, applying \eqref{t1b} and \eqref{tt},  we deduce that
$$x_{k_2}^1=x_{\sigma(k_2)}^2\mathcal R x_{\sigma(k_1)}^2=x_{k_1}^1.$$
It follows that $x_{k_2}^1\mathcal R x_{k_1}^1$ and, applying again \eqref{t1b}, we have $x_{k_1}^1\mathcal R x_{k_2}^1$. Recalling that $\mathcal R$ is a partial order binary relation  on the set $\{x_1^1,\ldots,x_{N_1}^1\}\cup\{x_1^2,\ldots,x_{N_2}^2\} $, we obtain $x_{k_1}^1= x_{k_2}^1$ which contradicts the fact that 
$$B(x_{k_1}^1,r_{k_1}^1)\cap B(x_{k_2}^1,r_{k_2}^1)=\emptyset.$$
This proves that the one-to-one map $\sigma$ of $\{1,\ldots,N\}$ appearing in \eqref{tt} is necessary the identity map and we have  
\begin{equation}\label{t1i}x_k^1=x_k^2,\quad (b_k^1)^{-2}-b_0^{-2})|B_{r_k^1}|=\lambda_k^1=\lambda_k^2=(b_k^2)^{-2}-b_0^{-2})|B_{r_k^2}|,\quad k=1,\ldots,N.\end{equation}
On the other hand, in view of  \eqref{t1a}, we have either
$b_k^1=b_k^2$ or $|B_{r_k^1}|=|B_{r_k^2}|$. Then, condition \eqref{t1i} implies that
$$x_k^1=x_k^2,\quad b_k^1=b_k^2,\quad r_k^1=r_k^2,\quad k=1,\ldots,N,$$
which clearly implies that $c_1=c_2$. Combining this with \cite[Theorem 8.1]{KU}, we deduce that we have also $f_1=f_2$ which completes the proof of the corollary.\qed

\section{Proof of of Theorem \ref{t2}}\label{sect:5}

This section is devoted to the proof of Theorem \ref{t2}.  In a similar way to Theorem \ref{t1} we know that, for any $\phi\in H^2(B_{R_0})$ satisfying $-\Delta\phi=0$ in $B_{R_0}$, \eqref{t1f} holds true.
On the other hand, in view of \eqref{cond1}-\eqref{cond3}, we have
$$ \begin{aligned}\int_{\Omega}c^{-2}_j\phi dx&=b_0^{-2} \int_{\Omega}\phi dx+((b_1)^{-2}-b_0^{-2}))\int_{\omega^j}\phi(x)dx\\
\ &=b_0^{-2} \int_{\Omega}\phi dx+((b_1)^{-2}-b_0^{-2}))\left(\sum_{k=1}^{m_j}\int_{B(x_p^j,r_p^j)}\phi(x)dx-\sum_{\ell=1}^{n_j}\int_{B(y_\ell^j,s_\ell^j)}\phi(x)dx\right)\end{aligned}$$
and applying the mean value theorem for harmonic functions, we find
$$ \int_{\Omega}c^{-2}_j\phi dx=b_0^{-2} \int_{\Omega}\phi dx+((b_1)^{-2}-b_0^{-2}))\left(\sum_{k=1}^{m_j}|B_{r_k^j}|\phi(x_k^j)-\sum_{\ell=1}^{n_j}|B_{s_\ell^j}|\phi(y_\ell^j)\right).$$
Combining this with \eqref{t1f} using the fact that $b_1\neq b_0$, for any $\phi\in H^2(B_{R_0})$ satisfying $-\Delta\phi=0$ in $B_{R_0}$, we obtain
\bel{t2e}\left(\sum_{k=1}^{m_1}|B_{r_k^1}|\phi(x_k^1)-\sum_{\ell=1}^{n_1}|B_{s_\ell^1}|\phi(y_\ell^1)\right)-\left(\sum_{k=1}^{m_2}|B_{r_k^2}|\phi(x_k^2)-\sum_{\ell=1}^{n_2}|B_{s_\ell^2}|\phi(y_\ell^2)\right)=0.\ee
For $j=1,2$, let us consider $z_1^j,\ldots,z_{m_j+n_j}^j\in\Omega$ and $\lambda_1^j,\ldots,\lambda_{m_j+n_j}^j\in\R$ such that
$$z_k^j=x_k^j,\quad z_{m_j+\ell}^j=y_{\ell}^j,\quad \lambda_k^j=|B_{r_k^j}|,\quad \lambda_{m_j+\ell}^j=-|B_{s_\ell^j}|,\quad k=1,\ldots,m_j,\ \ell=1,\ldots,n_j.$$
Applying \eqref{t2e}, for any $\phi\in H^2(B_{R_0})$ satisfying $-\Delta\phi=0$ in $B_{R_0}$, we get
\bel{t2f}\sum_{k=1}^{m_1+n_1}\lambda_k^1\phi(z_k^1)-\sum_{k=1}^{m_2+n_2}\lambda_k^2\phi(z_k^2)=0.\ee
For $j=1,2$, let us consider $v_j\in L^2(\Omega)$, $j=1,2$, the solution in the transposition sense of the boundary value problem
\begin{equation}\label{t2g}\left\{\begin{array}{ll}-\Delta v_j=\sum_{k=1}^{m_j+n_j}\lambda_k^j\delta_{z_k^j},\quad &\textrm{in}\ B_{R_0},\\  v_j=0, &\textrm{on }\partial B_{R_0}.\end{array}\right.\end{equation}
In a similar way to the proof of Theorem \ref{t1}, we deduce that \eqref{t2f} implies that
 $$\partial_\nu v_1(x)=\partial_\nu v_2(x),\quad x\in\partial B_{R_0} $$  and, applying Theorem \ref{t3}, we deduce that $m_1+n_1=m_2+n_2=N$ and, there exists a one-to-one map $\sigma$ of $\{1,\ldots,N\}$ such that 
\bel{t2h}z_k^1=z_{\sigma(k)}^2,\quad \lambda_k^1=\lambda_{\sigma(k)}^2,\quad k=1,\ldots,N.\ee
Recalling that 
$$\lambda_k^j=|B_{r_k^j}|>0,\quad  \lambda_{m_j+\ell}^j=-|B_{s_\ell^j}|<0,\quad k=1,\ldots,m_j,\ \ell=1,\ldots,n_j,$$
and using \eqref{t2h}, we deduce that $\sigma(\{1,\ldots,m_1\})=\{1,\ldots,m_2\}$ and $\sigma(\{m_1+1,\ldots,n_1+m_1\})=\{m_2+1,\ldots,m_2+n_2\}$. Therefore, we have $m_1=m_2=m$, $n_1=n_2=n$, $\sigma_1= \sigma|_{\{1,\ldots,m\}}$ is a  one-to-one map  of $\{1,\ldots,m\}$ and $\sigma_2= \sigma|_{\{m+1,\ldots,m+n\}}$ is a  one-to-one map  of $\{m+1,\ldots,m+n\}$. Combining this with \eqref{t2h}, we get
$$x_k^1=z_k^1=z_{\sigma_1(k)}^2=x_{\sigma_1(k)}^2,\quad |B_{r_k^1}|=\lambda_k^1=\lambda_{\sigma_1(k)}^2=|B_{r_{\sigma_1(k)}^2}|,\quad k=1,\ldots,m,$$
$$y_\ell^1=z_{m+\ell}^1=z_{\sigma_2(m+\ell)}^2=y_{\sigma_2(m+\ell)-m}^2,\quad -|B_{s_\ell^1}|=\lambda_{m+\ell}^1=\lambda_{\sigma_2(m+\ell)}^2=-|B_{s_{\sigma_2(m+\ell)-m}^2}|,\quad \ell=1,\ldots,n.$$
It follows that
$$x_k^1=x_{\sigma_1(k)}^2,\quad r_k^1=r_{\sigma_1(k)}^2,\quad y_\ell^1=y_{\sigma_2(m+\ell)-m}^2,\quad s_\ell^1=s_{\sigma_2(m+\ell)-m}^2,\quad k=1,\ldots,m,\ \ell=1,\ldots,n,$$
which clearly implies that $\omega^1=\omega^2$ and $c_1=c_2$. Combining this with \cite[Theorem 8.1]{KU}, we deduce that we have also $f_1=f_2$ which completes the proof of the theorem.

\section{Proof of Theorem \ref{t4}}\label{sect:6}

We start by proving Theorem \ref{t4} under the assumption that for all $j=1,2$ and $k=1,\ldots,N_j$, we have $x_k^j\neq y_k^j$. Without loss of generality we may assume that $N_1\leq N_2$ which implies that $|\{x_1^1,\ldots,x_{N_1}^1,y_1^1,\ldots,y_{N_1}^1\}|=2N_1$.
In a similar way to Theorem \ref{t1} we know that, for any $\phi\in H^2(B_{R_0})$ satisfying $-\Delta\phi=0$ in $B_{R_0}$, \eqref{t1f} holds true. Moreover,  \eqref{om2} implies
$$ \begin{aligned}\int_{\Omega}c^{-2}_j\phi dx&=b_0^{-2} \int_{\Omega}\phi dx+\sum_{k=1}^{N_j}((b_k^j)^{-2}-b_0^{-2}))\int_{\omega_k^j}\phi(x)dx\\
\ &=b_0^{-2} \int_{\Omega}\phi dx+\sum_{k=1}^{N_j}((b_k^j)^{-2}-b_0^{-2}))\left(\int_{B(x_k^j,r)}\phi(x)dx-\int_{B(y_k^j,s)}\phi(x)dx\right)\end{aligned}$$
and applying the mean value theorem for harmonic functions, we find
$$ \int_{\Omega}c^{-2}_j\phi dx=b_0^{-2} \int_{\Omega}\phi dx+\sum_{k=1}^{N_j}((b_k^j)^{-2}-b_0^{-2}))\left(|B_{r}|\phi(x_k^j)-|B_{s}|\phi(y_k^j)\right).$$
It follows that 
\bel{t4c} \sum_{k=1}^{N_1}((b_k^1)^{-2}-b_0^{-2}))\left(|B_{r}|\phi(x_k^1)-|B_{s}|\phi(y_k^1)\right)-\sum_{k=1}^{N_2}((b_k^2)^{-2}-b_0^{-2}))\left(|B_{r}|\phi(x_k^2)-|B_{s}|\phi(y_k^2)\right)=0.\ee
For $j=1,2$, let us consider $z_1^j,\ldots,z_{2N_j}^j\in\Omega$ and $\lambda_1^j,\ldots,\lambda_{2N_j}^j\in\R$ such that
$$z_k^j=x_k^j,\quad z_{N_j+\ell}^j=y_{\ell}^j,\quad \lambda_k^j=((b_k^j)^{-2}-b_0^{-2}))|B_{r}|,\quad \lambda_{N_j+\ell}^j=-((b_k^j)^{-2}-b_0^{-2})|B_{s}|,\quad k,\ell=1,\ldots,N_j.$$
Applying \eqref{t4c}, for any $\phi\in H^2(B_{R_0})$ satisfying $-\Delta\phi=0$ in $B_{R_0}$, we get
\bel{t4d}\sum_{k=1}^{2N_1}\lambda_k^1\phi(z_k^1)-\sum_{k=1}^{2N_2}\lambda_k^2\phi(z_k^2)=0.\ee
For $j=1,2$, let us consider $v_j\in L^2(\Omega)$, $j=1,2$, the solution in the transposition sense of the boundary value problem
$$\left\{\begin{array}{ll}-\Delta v_j=\sum_{k=1}^{2N_j}\lambda_k^j\delta_{z_k^j},\quad &\textrm{in}\ B_{R_0},\\  v_j=0, &\textrm{on }\partial B_{R_0}.\end{array}\right.$$
In a similar way to the proof of Theorem \ref{t1}, we deduce that \eqref{t4d} implies that
 $$\partial_\nu v_1(x)=\partial_\nu v_2(x),\quad x\in\partial B_{R_0} $$  and, applying Theorem \ref{t3}, we deduce that 
$$2N_1=|\{x_1^2,\ldots,x_{N_2}^1,y_1^2,\ldots,y_{N_2}^2\}|\leq 2N_2$$
and, there exists a one-to-one map $\sigma$ of $\{1,\ldots,2N_1\}$ such that 
$$z_k^1=z_{\sigma(k)}^2,\quad \lambda_k^1=\lambda_{\sigma(k)}^2,\quad k=1,\ldots,2N_1.$$
In addition, recalling that, for all $k,\ell=1,\ldots,N_j$, we have
\bel{tete}\lambda_k^j=((b_k^j)^{-2}-b_0^{-2})^2|B_{r}|>0,\quad  \lambda_{N_j+\ell}^j=-((b_k^j)^{-2}-b_0^{-2})^2|B_{s}|<0,\ee
we deduce that 
$$\{x_1^1,\ldots,x_{N_1}^1\}=\{z_1^1,\ldots,z_{N_1}^1\}=\{z_1^2,\ldots,z_{N_2}^2\}=\{x_1^2,\ldots,x_{N_2}^2\},$$
$$\{y_1^1,\ldots,y_{N_1}^1\}=\{z_{N_1+1}^1,\ldots,z_{2N_1}^1\}=\{z_{N_2+1}^2,\ldots,z_{2N_2}^2\}=\{y_1^2,\ldots,y_{N_2}^2\}.$$
It follows that $N=N_1=N_2$ and there exists $\sigma_1$, $\sigma_2$ two one-to-one map  of $\{1,\ldots,N\}$ such that
$$x_k^1=x_{\sigma_1(k)}^2,\quad \lambda_k^1=\lambda_{\sigma_1(k)}^2,\quad y_k^1=y_{\sigma_2(k)}^2,\quad \lambda_{N+k}^1=\lambda_{N+\sigma_2(k)}^2,\quad k=1,\ldots,N.$$
Combining this with \eqref{tete}, we deduce that $\sigma_1=\sigma_2$ and it follows that 
$$x_k^1=x_{\sigma_1(k)}^2,\quad b_k^1=b_{\sigma_1(k)}^2,\quad y_k^1=y_{\sigma_1(k)}^2,\quad k=1,\ldots,N$$
which clearly  implies that $c_1=c_2$. Combining this with \cite[Theorem 8.1]{KU}, we deduce that we have also $f_1=f_2$ which completes the proof of the theorem.

Now let us assume that there exist $k_0\in\{1,\ldots,N\}$ and $j_0\in\{1,2\}$ such that $x_{k_0}^{j_0}=y_{k_0}^{j_0}$. Repeating the above argumentation, we can show that there exists $k_1,k_2\in \{1,\ldots,N\}$ $x_{k_1}^{1}=y_{k_1}^{1}=x_{k_2}^{2}=y_{k_2}^{2}$. Therefore, repeating the above argumentation, we can prove that $c_1=c_2$ and we can complete the proof of the theorem.

\section{Stability estimates}\label{sect:7}

This section is devoted to the stability estimates stated in  Theorem \ref{t8}. We start with the following improved regularity result for solutions of \eqref{eq1}.

\begin{lem}\label{l6} Let $f\in H^3(\R^3)$ and  $c$ be a piecewise constant wave speed coefficient of the form \eqref{c}. Then, for any $s\in(0,\frac 1 2)$, the solution $u$ of \eqref{eq1} is lying in $C^2([0,+\infty);H^{s}(\R^3))$ and, for all $T>0$,
\bel{l6a}\norm{u}_{C^2([0,T];H^{s}(\R^3))}\leq C\norm{f}_{H^3(\R^3)},\ee
with $C>0$ depending on $s$, $M_1\geq \norm{c}_{L^\infty(\R^3)}+\norm{c^{-1}}_{L^\infty(\R^3)}$,  $R_0$ and $T$.\end{lem}
\begin{proof} Fix $s\in(0,\frac 1 2)$.
Let us first observe that, in view of Lemma \ref{l11} in the Appendix, problem \eqref{eq1} admits a unique solution $u\in C^2([0,+\infty);L^2(\R^3))\cap C^1([0,+\infty);H^1(\R^3))\cap C([0,+\infty);H^2(\R^3))$. Let us fix $v=\partial_tu$ and observe that $v$ solves the initial value problem
\begin{equation}\label{eq11}\left\{\begin{array}{ll}c^{-2}(x)\partial_t^2v-\Delta v=0,\quad &\textrm{in}\ \R_+\times\R^3,\\  v(0,x)=0,\quad \partial_tv(0,x)=g(x),&x\in\R^3,\end{array}\right.\end{equation}
with $g=c^2\Delta f$. Since $f\in H^3(\R^3)$, applying \cite[Theorem 11.4, Chapter I]{LM1}, we deduce that $c^2\Delta f\in H^{s}(\R^3)$ and
\bel{l6b} \begin{aligned}\norm{c^2\Delta f}_{H^{s}(\R^3)}&\leq b_0\norm{\Delta f}_{H^{s}(\R^3)}+\sum_{k=1}^N|b_k-b_0|\norm{\Delta f\mathds{1}_{B(x_k,r_k)}}_{H^{s}(\R^3)}\\
&\leq C\norm{\Delta f}_{H^{s}(\R^3)}\leq C\norm{\Delta f}_{H^{1}(\R^3)}\leq C\norm{f}_{H^3(\R^3)},\end{aligned}\ee
with $C>0$ depending only on $M_1$ and $R_0$. On the other hand,  for $g\in H^\ell(\R^3)$, $\ell=0,1$, applying Lemma \ref{l11} of the Appendix combined with the estimates in \cite[Theorem 8.1, Chapter 3]{LM1} and \cite[Theorem 8.2, Chapter 3]{LM1}, we have $v\in C^{1}([0,+\infty);H^\ell(\R^3))$ and, for all $T>0$, there exists $C>0$ depending on $M_1$ and $T$ such that
$$\norm{v}_{C^{1}([0,T];H^\ell(\R^3))}\leq C\norm{g}_{H^\ell(\R^3)}.$$
In view of \cite[Theorem 5.1, Chapter I]{LM1} and \cite[Theorem 7.1, Chapter I]{LM1}, by interpolation we deduce that , for $g\in H^{s}(\R^3)$,  we have $v\in C^{1}([0,+\infty);H^{s}(\R^3))$ and, for all $T>0$, there exists $C>0$ depending on $M_1$ and $T$ such that
$$\norm{v}_{C^{1}([0,T];H^{s}(\R^3))}\leq C\norm{g}_{H^{s}(\R^3)}.$$
Applying this result to $g=c^2\Delta f$, we deduce that $u\in C^2([0,+\infty);H^{s}(\R^3))$ and \eqref{l6a} holds true.

\end{proof}
Armed with this lemma we are now in position to complete the proof of Theorem \ref{t8}.\\

 \textbf{Proof of of Theorem \ref{t8}.}
 We start by proving \eqref{esti14}. Let us first observe that in view of Theorem \ref{l111}, the map $p\mapsto\widehat{u_j}(p,\cdot)|_{B_{R_0}}$, $j=1,2$, can be extended analytically to an element of $\mathcal H(\mathbb D_\delta;H^2(B_{R_0}))$. In view of Lemma \ref{l1}, for $j,k=1,2$, we have
$$\partial_p^k\widehat{u_j}(p,x)|_{p=0}=k!u_j^{(k)},$$
with
$$u_j^{(2)}(x)=D_j:=-\frac{1}{2\pi b_0}\int_{\R^3}\frac{f_j(y)}{c_j(y)^2}dy,\quad x\in B_{R_0},$$
$$-\Delta u_j^{(4)}(x)=-c^{-2}_ju_1^{(2)}(x),\quad x\in B_{R_0}.$$
Fixing $u^{(4)}=u_1^{(4)}-u_2^{(4)}\in H^2(B_{R_0})$, we observe that $u^{(4)}$ solves the equation
$$-\Delta u^{(4)}(x)=c^{-2}_2u_2^{(2)}(x)-c^{-2}_1u_1^{(2)}(x)=(D_2-D_1)c_2^{-2}(x)+D_1(c^{-2}_2-c^{-2}_1)(x),\quad x\in B_{R_0}.$$
Fixing $v\in H^2(\Omega)$ a solution of $\Delta v=0$ on $\Omega$, multiplying the above equation by $v$, integrating by parts and applying the mean value theorem, we find
$$\begin{aligned}&-\int_{\partial \Omega}\partial_\nu u^{(4)}vd\sigma(x)+\int_{\partial \Omega} u^{(4)}\partial_\nu vd\sigma(x)\\
&=(D_2-D_1)\int_{\Omega }c^{-2}_{2}(x)v(x)dx+D_1\int_{\Omega }(c^{-2}_2-c^{-2}_1)vdx\\
&=(D_2-D_1)\int_{\Omega }c^{-2}_{2}(x)v(x)dx+D_1(b_1^2)^{-2}\int_{B(x_1^2, r_1)}vdx-D_1(b_1^1)^{-2}\int_{B(x_1^1, r_1)}vdx\\
&=(D_2-D_1)\int_{\Omega }c^{-2}_{2}(x)v(x)dx+D_1|B_{r_1}|((b_1^2)^{-2}v(x^2_1)-(b_1^1)^{-2}v(x^1_1)).\end{aligned}$$
Using the definition of $D_j$ and $u_j^{(4)}$, $j=1,2$, we obtain
\bel{esti8}\begin{aligned}&\abs{(b_1^1)^{-2}v(x^1_1)-(b_1^2)^{-2}v(x^2_1)}\\
&\leq C \left(\sum_{k=1}^{2}\sum_{\ell=0}^1\norm{\partial_p^{2k} \partial_\nu^\ell\widehat{u_1}(p,\cdot)|_{p=0}-\partial_p^{k} \partial_\nu^\ell\widehat{u_2}(p,\cdot)|_{p=0}}_{L^2(\partial \Omega)}\right)\norm{v}_{H^2(\Omega )},\end{aligned}\ee
with $C>0$ depending on $m$, $R_0$ and $M$.
For $\ell=1,2,3$, fix $v_\ell(x)=x_\ell-(x^1_1)_\ell$, $x=(x_1,x_2,x_3)\in \Omega $, with $x^j_1=((x_1^j)_1,(x_1^j)_2,(x_1^j)_3)$, $j=1,2$. Then, choosing $v=v_\ell$, $\ell=1,2,3$, in \eqref{esti8}, we obtain
$$(b_1^2)^{-2}\abs{(x_1^2)_\ell^2-(x_1^1)_\ell}\leq C \left(\sum_{k=1}^{2}\sum_{\ell=0}^1\norm{\partial_p^{2k} \partial_\nu^\ell\widehat{u_1}(p,\cdot)|_{p=0}-\partial_p^{k} \partial_\nu^\ell\widehat{u_2}(p,\cdot)|_{p=0}}_{L^2(\partial \Omega)}\right)$$
which implies that
\bel{esti9}\abs{x^2_1-x^1_1}\leq C \left(\sum_{k=1}^{2}\sum_{\ell=0}^1\norm{\partial_p^{2k} \partial_\nu^\ell\widehat{u_1}(p,\cdot)|_{p=0}-\partial_p^{k} \partial_\nu^\ell\widehat{u_2}(p,\cdot)|_{p=0}}_{L^2(\partial \Omega)}\right),\ee
with $C>0$ depending on  $m$, $M$ and $R_0$.
Now let us assume that $x_1^1\neq x_1^2$ and let  us consider $\omega_1=\frac{x^1-x^2}{|x^1-x^2|}$, $\omega_2\in \mathbb S^{2}$ such that $\omega_2\cdot\omega_1=0$ and $\omega_3=\omega_1\times\omega_2$. Then, choosing $v(x)=\exp((x-x^1_1)\cdot\omega_1+i(x-x^1_1)\cdot\omega_2)$ in \eqref{esti8}, we obtain
\bel{esti10}\begin{aligned}\abs{b_1^1-b_1^2}&=\frac{(b_1^1)^{2}(b_1^2)^{2}}{b^1_1+b_1^2}\abs{(b_1^1)^{-2}v(x^1)-(b_1^2)^{-2}v(x^2)}\\
&\leq C \left(\sum_{k=1}^{2}\sum_{\ell=0}^1\norm{\partial_p^{2k} \partial_\nu^\ell\widehat{u_1}(p,\cdot)|_{p=0}-\partial_p^{k} \partial_\nu^\ell\widehat{u_2}(p,\cdot)|_{p=0}}_{L^2(\partial \Omega)}\right).\end{aligned}\ee
Similarly, we can deduce \eqref{esti10} when $x_1^1= x_1^2$, which proves that \eqref{esti10}  is always true.
Thus, for all $1\leq q<\infty$, we obtain
$$\begin{aligned}&\norm{c_1-c_2}_{L^{q}(\Omega)}\\
&\leq \norm{(b^1_1-b_0)\mathds{1}_{B(x^1,r)\setminus B(x^2,r)}}_{L^{q}(\Omega)}+\norm{(b^1_1-b^2_1)\mathds{1}_{B(x^1,r)\cap B(x^2,r)}}_{L^{q}(\Omega)}+\norm{(b^1_1-b_0)\mathds{1}_{B(x^2,r)\setminus B(x^1,r)}}_{L^{q}(\Omega)}\\
&\leq |b^1-b_0||B(x^1,r)\setminus B(x^2,r)|^{\frac{1}{q}}+|b^1-b^2||B(x^1,r)\cap  B(x^2,r)|^{\frac{1}{q}}+|b^2-b_0||B(x^1,r)\setminus B(x^2,r)|^{\frac{1}{q}}\\
&\leq C(|B(x^1,r)\setminus B(x^2,r)|^{\frac{1}{q}}+|b^1-b^2|+|B(x^1,r)\setminus B(x^2,r)|^{\frac{1}{q}}).\end{aligned}$$
with $C>0$ depending on  $M$, $R_0$.
Now let us observe that $B(x^1,r)\setminus B(x^2,r)\subset B(x^2,r+|x^1-x^2|)\setminus B(x^2,r)$ and
$$|B(x^1,r)\setminus B(x^2,r)|\leq |B(x^2,r+|x^1-x^2|)\setminus B(x^2,r)|=c((r+|x^1-x^2|)^3-r^3)\leq C|x^1-x^2|,$$
with $C>0$ depending only on $R_0$. In the same way, we have $|B(x^1,r)\setminus B(x^2,r)|\leq C|x^1-x^2|$ and, for all $1\leq q<\infty$, it follows
\bel{esti115} \norm{c_1-c_2}_{L^{q}(\Omega)}\leq C(|x^1-x^2|^{\frac{1}{q}}+|b^1-b^2|),\ee
with $C>0$ depending on $q$, $R_0$, $m$ and $M$. Applying \eqref{esti9}-\eqref{esti10}, we obtain \eqref{esti14}.

Fix $s\in(0,\frac{1}{2})$ and let us show the estimate \eqref{t8a}.
Applying estimate \eqref{esti115}, with $q=q_s=\frac{6}{3+2s}$, we obtain
\bel{esti15} \norm{c_1-c_2}_{L^{q_s}(\Omega )}\leq C\left(\sum_{k=1}^{2}\sum_{\ell=0}^1\norm{\partial_p^{2k} \partial_\nu^\ell\widehat{u_1}(p,\cdot)|_{p=0}-\partial_p^{k} \partial_\nu^\ell\widehat{u_2}(p,\cdot)|_{p=0}}_{L^2(\partial \Omega)}\right)^{\frac{1}{q_s}}.\ee
We fix $u=u_1-u_2$, $g=u_1-u_2$ on $[0,+\infty)\times\partial \Omega $ and $T>4R_0\frac{b_0}{m}$. Recalling that $u_1-u_2\in H^2((0,T)\times\R^3)$, supp$(f_j)\subset \Omega $ and applying \cite[Chapter I, Theorem 2.3]{LM2}, we can find $G\in H^2((0,T)\times \Omega )$ such that
\bel{esti16} G|_{(0,T)\times\partial \Omega }=g|_{(0,T)\times\partial \Omega },\quad G(0,\cdot)=\partial_tG(0,\cdot)\equiv0,\ee
\bel{esti18}\norm{G}_{H^2((0,T)\times \Omega )}\leq C(\|  u_1 -   u_2\|_{H^{\frac{3}{2}}((0, T)\times\partial \Omega   )}+\|  t^{-\frac{1}{2}}(\partial_t u_1 -   \partial_tu_2)\|_{L^2((0, T)\times\partial \Omega   )}).\ee 
Moreover, one can check that we can split $u|_{(0,T)\times\Omega}$ into two terms $u|_{(0,T)\times\Omega}=w+G$ where $w$ solves the initial  boundary value problem 
\begin{equation}\label{eq12}\left\{\begin{array}{ll}c_1^{-2}\partial_t^2w-\Delta_x w=-(c_1^{-2}\partial_t^2-\Delta)G(t,x)+(c^{-2}_2-c_1^{-2})\partial_t^2u_2,\quad &\textrm{in}\ (0,T)\times \Omega ,\\  w(0,\cdot)=f_1-f_2,\quad \partial_tw(0,\cdot)=0,\quad &\textrm{in}\ \Omega ,\\  w=0,\quad &\textrm{on}\ (0,T)\times\partial \Omega .\end{array}\right.\end{equation}
Let us now recall an observability inequality   for solutions $v$ of the following initial boundary value problem
\begin{equation}\label{eqq2}\left\{\begin{array}{ll}c_1^{-2}\partial_t^2v-\Delta v=F(t,x),\quad &\textrm{in}\ (0,T)\times\Omega,\\  v(0,\cdot)=v_0,\quad \partial_tv(0,\cdot)=v_1,\quad &\textrm{in}\ \Omega,\\  v=0,\quad &\textrm{on}\ (0,T)\times\partial\Omega,\end{array}\right.\end{equation}
with $F\in L^2((0,T)\times\Omega)$, $v_0\in H^1_0(\Omega)$, $v_1\in L^2(\Omega)$.
\begin{prop}\label{p4} \emph{(Proposition 4.1, \cite{KT})} There exists  $C>0$ depending on $\Omega$, $T$, $b_0$ and $m$, such that, for $v$ the solution of \eqref{eqq2} , we have 
\bel{pp3a} \int_\Omega [|v_1|^2+c_1^2|\nabla v_0|^2]dx\leq C\left(\int_0^T\int_{\partial\Omega}|\partial_\nu v(t,x)|^2d\sigma(x)dt +\int_0^T\int_{\Omega}|F(t,x)|^2dxdt\right).
\ee
\end{prop}
This result was proved in \cite[Proposition 4.1]{KT} with $b_0=1$ and $b_1^1<1$. Nevertheless, using the fact that $b_1^1<b_0$, it can be easily extended to solutions of \eqref{eqq2}, by applying the result of \cite[Proposition 4.1]{KT} to the map
$$v_1(t,x)=v(t/b_0,x),\quad (t,x)\in[0,b_0T]\times\Omega$$
which solves the initial boundary value problem
$$\left\{\begin{array}{ll}\left(\frac{c_1}{b_0}\right)^{-2}\partial_t^2v_1-\Delta v_1=F(t/b_0,x),\quad &\textrm{in}\ (0,b_0T)\times\Omega,\\  v_1(0,\cdot)=v_0,\quad \partial_tv_1(0,\cdot)=b_0^{-1}v_1,\quad &\textrm{in}\ \Omega,\\  v_1=0,\quad &\textrm{on}\ (0,b_0T)\times\partial\Omega.\end{array}\right.$$
Applying Proposition \ref{p4} to the solution $w$ of \eqref{eq12}, we obtain the following estimate
$$\begin{aligned}&\norm{f_1-f_2}_{H^1(\Omega )}\\
&\leq C\left(\norm{-(c_1^{-2}\partial_t^2-\Delta)G+(c^{-2}_2-c_1^{-2})\partial_t^2u_2}_{L^2((0,T)\times \Omega )}+\norm{\partial_\nu w}_{L^2((0,T)\times\partial \Omega )}\right)\\
&\leq C\left(\norm{G}_{H^2((0,T)\times \Omega )}+\norm{(c^{-2}_2-c_1^{-2})\partial_t^2u_2}_{L^2((0,T)\times \Omega )}+\norm{\partial_\nu u-\partial_\nu G}_{L^2((0,T)\times\partial \Omega )}\right) \\
&\leq C\left(\norm{G}_{H^2((0,T)\times \Omega )}+\norm{(c^{-2}_2-c_1^{-2})\partial_t^2u_2}_{L^2((0,T)\times \Omega )}+\norm{\partial_\nu u}_{L^2((0,T)\times\partial \Omega )}+\underbrace{\norm{\partial_\nu G}_{L^2((0,T)\times\partial \Omega )}}_{\leq C\norm{G}_{H^2((0,T)\times \Omega )}}\right)\end{aligned}.$$
Combining this with Lemma \ref{l6}, \eqref{esti14}, \eqref{esti18} and the Sobolev embedding theorem, we obtain
$$\begin{aligned}&\norm{f_1-f_2}_{H^1(\Omega )}\\
&\leq C\left(\|  u_1 -   u_2\|_{H^{\frac{3}{2}}((0, T)\times\partial \Omega   )}+\|  t^{-\frac{1}{2}}(\partial_t u_1 -   \partial_tu_2)\|_{L^2((0, T)\times\partial \Omega   )}\right)\\
&\ \ \ \ +C(\norm{c_1-c_2}_{L^{q_s}(\Omega )}\norm{\partial_t^2u_2}_{L^\infty(0,T;L^{\frac{6}{3-2s}}( \R^3))}+\norm{\partial_\nu u_1-\partial_\nu u_2}_{L^2((0,T)\times\partial \Omega )})\\
&\leq C\left(\|  u_1 -   u_2\|_{H^{\frac{3}{2}}((0, T)\times\partial \Omega   )}+\|  t^{-\frac{1}{2}}(\partial_t u_1 -   \partial_tu_2)\|_{L^2((0, T)\times\partial \Omega   )}+\norm{c_1-c_2}_{L^{q_s}(\Omega )}\norm{u_2}_{C^2([0,T];H^{s}( \R^3))}\right)\\
&\ \ \ \ +C\norm{\partial_\nu u_1-\partial_\nu u_2}_{L^2((0,T)\times\partial \Omega )}\\
&\leq C\left(\|  u_1 -   u_2\|_{H^{\frac{3}{2}}((0, T)\times\partial \Omega   )}+\|  t^{-\frac{1}{2}}(\partial_t u_1 -   \partial_tu_2)\|_{L^2((0, T)\times\partial \Omega   )}+\norm{\partial_\nu u_1-\partial_\nu u_2}_{L^2((0,T)\times\partial \Omega )}\right)\\
&\ \ \ \ +C\left(\sum_{k=1}^{2}\sum_{\ell=0}^1\norm{\partial_p^{2k} \partial_\nu^\ell\widehat{u_1}(p,\cdot)|_{p=0}-\partial_p^{k} \partial_\nu^\ell\widehat{u_2}(p,\cdot)|_{p=0}}_{L^2(\partial \Omega)}\right)^{\frac{3+2s}{6}},\end{aligned}$$
with $C>0$ depending on $T$, $s$, $b_0$, $\Omega$, $m$, $M$ and $R_0$.
Form this estimate we deduce \eqref{t8a} which completes the proof of Theorem \ref{t8}.\qed

\section{Appendix}
\subsection{Well-posedness of problem \eqref{eq1}}

In this section we consider the well-posedness of the problem 
\begin{equation}\label{eqqq1}
\left\{
\begin{array}{ll}
c^{-2}(x)\partial_t^2u(t, x) - \Delta u(t, x) = 0, & (t, x)\in \mathbb{R}_+ \times \mathbb{R}^3, \\
u(0,x) = f(x), \quad \partial_tu(0,x) = g(x), & x \in \mathbb{R}^3.
\end{array}
\right.
\end{equation}

Our result can be stated as follows.

\begin{lem}\label{l11} Let $c \in L^\infty(\mathbb{R}^3)$ be a strictly positive piecewise constant function such that \eqref{ccc} is fulfilled and $c$ is constant on $\mathbb{R}^3 \setminus \overline{\Omega}$. For $f\in H^1(\R^3)$ and $g\in L^2(\R^3)$,  \eqref{eqqq1} admits a unique solution $u\in  C^1([0,+\infty);L^2(\R^3))\cap  C([0,+\infty);H^1(\R^3))$. Moreover, 
for $f \in H^2(\mathbb{R}^3)$ and $g\in H^1(\R^3)$, the
 problem \eqref{eqqq1} admits a unique solution $u\in  C^2([0,+\infty);L^2(\R^3))\cap C^1([0,+\infty);H^1(\R^3))\cap C([0,+\infty);H^2(\R^3))$. 
\end{lem}

\begin{proof} 

These results are classical but we recall them for sake of completeness.  We consider again the operator $L=-c^2\Delta$ acting on $L^2(\R^3;c^{-2}dx)$ with domain $D(L)=H^2(\R^3)$ and, following \cite[Proposition A.1]{Sh}, we recall that $L$ is the unique selfadjoint operator  associated with the closed, symmetric,  positive quadratic form $a$, acting on $L^2(\R^3;c^{-2}dx)$ with the form domain $D(a):=H^1(\R^3)$, defined by
$$a(h_1,h_2)=\int_{\R^3}\nabla h_1\cdot\overline{\nabla h_2}dx,\quad h_1,h_2\in H^1(\R^3). $$
Using the operator $L$, we can rewrite the problem \eqref{eqqq1} as follows
\begin{equation}\label{eqq1}
\left\{
\begin{array}{ll}
u''(t) + L u(t) = 0, & t\in \mathbb{R}_+ , \\
u(0) = f, \quad u'(0) = g. & 
\end{array}
\right.
\end{equation}
We start with the first claim of the lemma. Fixing $H=L^2(\R^3;c^{-2}dx)$, $V=D(a)=H^1(\R^3)$ and applying \cite[Theorem 8.1, Chapter 3]{LM1} and \cite[Theorem 8.2, Chapter 3]{LM1}, we deduce that for $f\in H^1(\R^3)$ and $g\in L^2(\R^3)$, \eqref{eqq1} admits a unique solution  $u\in  C^1([0,+\infty);L^2(\R^3))\cap  C([0,+\infty);H^1(\R^3))$. This proves the first claim of the lemma. 

Now let us consider the second claim. For this purpose, consider $f\in H^2(\R^3)= D(L)$ and $g\in H^1(\R^3)$, and, applying the first claim, let us consider $u\in  C^1([0,+\infty);L^2(\R^3))\cap  C([0,+\infty);H^1(\R^3))$ the unique solution of \eqref{eqq1}. Notice that $v=u'$ solves the problem
\begin{equation}\label{eqq2}
\left\{
\begin{array}{ll}
v''(t) + L v(t) = 0, & t\in \mathbb{R}_+ , \\
v(0) = g, \quad v'(0) = -Lf. & 
\end{array}
\right.
\end{equation}
Therefore, repeating the above argumentation, we deduce that $v\in C^1([0,+\infty);L^2(\R^3))\cap  C([0,+\infty);H^1(\R^3))$. This implies that $u\in  C^2([0,+\infty);L^2(\R^3))\cap  C^1([0,+\infty);H^1(\R^3))\cap C([0,+\infty);D(L))$. Recalling that  
$D(L)=H^2(\R^3)$, this completes the proof of the  lemma.

\end{proof}

\subsection{Point source identification}
In all this section, we assume that $\mathcal O$ is a $C^2$ bounded and connected domain in $\R^3$. Let $a:=(a_{i,j})_{1 \leq i,j \leq 3} \in C^2(\overline{\mathcal O};\R^{3^2})$, be symmetric:
$$ a_{i,j}(x)=a_{j,i}(x),\ x \in \overline{\mathcal O},\ i,j = 1,2,3, $$
and satisfying the ellipticity condition
\bel{a1}
\exists b>0,\ \sum_{i,j=1}^3 a_{i,j}(x) \xi_i \xi_j \geq b |\xi|^2,\ x \in \overline{\mathcal O},\ \xi=(\xi_1,\xi_2,\xi_3) \in \R^3.
\ee
 We define the operator $\mathcal A$ by
$$ 
\mathcal A v(x) :=-\sum_{i,j=1}^3 \partial_{x_i} 
\left( a_{i,j}(x) \partial_{x_j} v(x) \right),\  x\in \mathcal O. 
$$ 
Fixing $x_1,\ldots,x_N$, $N$ distinct points in $\mathcal O$ and $\lambda_1,\ldots,\lambda_N\in \R\setminus\{0\}$, we consider the solution in the transposition sense of the boundary value problem
\begin{equation}\label{bvp1}\left\{\begin{array}{ll}\mathcal A v=\sum_{k=1}^{N}\lambda_k\delta_{x_k},\quad &\textrm{in}\ \mathcal O,\\  v=0, &\textrm{on}\ \partial \mathcal O.\end{array}\right.\end{equation}
We denote by $\partial_{\nu_a}$  the conormal derivative associated with the coefficient $a$ of $\partial \mathcal O$ defined by
$$\partial_{\nu_a}v(x)=\sum_{j=1}^3a_{ij}(x)\partial_{x_j}v(x)\nu_i(x),\quad x\in\partial \mathcal O,$$
where $\nu=(\nu_1,\nu_2,\nu_3)$ denotes the outward unit normal vector to $\partial\mathcal O$.
We show the following result about unique recovery of point sources of elliptic equations as well as properties of the solution of \eqref{bvp1}.

\begin{Thm}\label{t3} 
For  all $r>0$, we denote by $\mathcal O_r$  the open subset of $B_{R_0}$ defined by 
$$\mathcal O_r:=\mathcal O\setminus\cup_{k=1}^{N} \overline{B(x_k,r)}.$$
Problem \eqref{bvp1} admits a unique solution in the transposition sense lying in $L^2(\mathcal O)$ and, choosing $r>0$ sufficiently small, we have $\mathcal O_r\neq\emptyset$, $\partial\mathcal O\subset\partial\Omega_r$  and  $v|_{\mathcal O_r}\in H^2(\mathcal O_r)$. In addition, for any $w\in H^2(\mathcal O)$ we have
\begin{equation}\label{t3a}\sum_{k=1}^{N}\lambda_kw(x_k)=- \int_{\partial \mathcal O}\partial_{\nu_a} v wd\sigma(x)+\int_{ \mathcal O} v \mathcal A wdx.  \end{equation}
In addition, for j=1,2, fixing $S$ an arbitrary not empty open subset of $\partial \mathcal O$ and $v_j$ the solution of \eqref{bvp1}, with $N=N_j$, $x_k=x_k^j$, $\lambda_k=\lambda_k^j$, the condition
\begin{equation}\label{t3b}  \partial_\nu v_1(x)=\partial_\nu v_2(x),\quad x\in S\end{equation}
implies that $N_1=N_2=N$ and, there exists a one-to-one map $\sigma$ of $\{1,\ldots,N\}$ such that
\bel{t3c}x_k^1=x_{\sigma(k)}^2,\quad \lambda_k^1=\lambda_{\sigma(k)}^2,\quad k=1,\ldots,N.\ee
\end{Thm}

\begin{proof} The proof of these results for \eqref{bvp1} with $\mathcal A=-\Delta$ can be deduced by following the argumentation in \cite{EH1,EH2}. We give here an extension of such arguments to more general class of elliptic operators with variable coefficients. Let us first observe that by the Sobolev embedding theorem $H^2(\mathcal O)$ embedded continuously into $C(\overline{\mathcal O})$ and therefore \eqref{bvp1} admits a unique solution in the transposition sense $v$ which will be the unique element $v\in L^2(\mathcal O)$ satisfying the condition
$$\int_{\mathcal O}v\mathcal A wdx=\sum_{k=1}^{N}\lambda_kw(x_k),\quad w\in H^1_0(\mathcal O)\cap H^2(\mathcal O).$$
Now let us fix $r>0$ sufficiently small such that $\mathcal O_r\neq\emptyset$, $\partial \mathcal O\subset\partial\mathcal O_r$. In order to prove that $v|_{\mathcal O_r}\in H^2(\mathcal O_r)$, let us fix  $R_1>R_2>0$ such that $\overline{\mathcal O}\subset B_{R_2}$ and consider the differential operator $\mathcal A$ acting on $B_{R_1}$ by extending its coefficients with $C^2$ regularity. By the theory of fundamental solutions of elliptic equations (see e.g. \cite[Theorem 3]{Ka}), there exists a function $P$ defined in $B_{R_1}^2\setminus D$, with $D:=\{(x,y)\in B_{R_1}^2:\ x=y\}$, satisfying the following properties:
\begin{itemize}
\item[(1)] For all $y\in B_{R_1}$, the map $x\mapsto P(x,y)\in C^2(B_{R_1}\setminus \{y\})$ and 
$$ |P(x,y)|\leq C|x-y|^{-1},\quad \mathcal A P(x,y)=0,\quad x\in B_{R_1}\setminus \{y\}.$$
\item[(2)] For all $\psi \in C^\infty_0(B_{R_1})$ we have
\bel{fond2}\int_{B_{R_1}}P(x,y)\mathcal A\psi(x)\,\d x=\psi(y).\ee
\end{itemize}
Then, for $k=1,\ldots,N$, fixing $w_k$ the $H^2(\mathcal O)$-solution  of the boundary value problem 
\begin{equation}\label{eq3}
\left\{\begin{aligned}
\mathcal{A} w_k &= 0, && \mbox{in }\mathcal O,\\
 w_k(x)&= -P(x,x_k), && x\in\partial \mathcal O,
\end{aligned}\right.
\end{equation}
one can check that the solution $v$ of \eqref{bvp1} takes the form
\bel{t3f}v(x)=\sum_{k=1}^N\lambda_k(P(x,x_k)+w_k(x)),\quad x\in \mathcal O,\ee
which proves that $v|_{\mathcal O_r}\in H^2(\mathcal O_r)$.

Now let us consider the identity \eqref{t3a}. Without loss of generality, we may assume that $N=1$ and $x_1=0$. Fix $w\in H^2(\mathcal O)$
and, notice that for $w_1$ the solution of \eqref{eq3}, we have
\begin{equation}\label{t3g}\int_{ \mathcal O}  \mathcal A w w_1dx= \int_{\partial \mathcal O}\partial_{\nu_a} w(x) P(x,0)d\sigma(x)+\int_{\partial \mathcal O}\partial_{\nu_a} w_1 wd\sigma(x).  \end{equation}
In addition,  by density, we can prove that \eqref{fond2} is still true for $\psi\in H^2_0(B_{R_1})$,  where $H^2_0(B_{R_1})$ denotes the closure of $C^\infty_0(B_{R_1})$ in $H^2(B_{R_1})$.  Extending $w$ into an element of $H^2_0(B_{R_1})$, still denoted by $w$, and satisfying $w=0$ on $B_{R_1}\setminus B_{R_2}$,  we deduce from \eqref{fond2} that
$$\int_{B_{R_2}}P(x,0)\mathcal A w(x)\,\d x=\int_{B_{R_1}}P(x,0)\mathcal A w(x)\,\d x=w(0).$$
Recalling that $x\mapsto P(x,0)\in C^2(\overline{B_{R_2}}\setminus \mathcal O)$ and, by integration by parts, we get
\begin{align*}
\int_{B_{R_2}\setminus \mathcal O}P(x,0)\mathcal A w(x)\,\d x&=\int_{\partial \mathcal O}P(x,0)\partial_{\nu_a}w(x)\,\d\sigma(x)-\int_{\partial \mathcal O}\partial_{\nu_a} P(x,0)w(x)\,\d\sigma(x)\\
&\ \ +\int_{B_{R_2}\setminus \mathcal O}\mathcal A P(x,0)w(x)\,\d x\\
&=\int_{\partial \mathcal O}P(x,0)\partial_{\nu_a}w(x)\,\d\sigma(x)-\int_{\partial \mathcal O}\partial_{\nu_a} P(x,0)w(x)\,\d\sigma(x).\end{align*}
Then, it follows that
\begin{align*}\int_{\mathcal O}P(x,0)\mathcal A w(x)\,\d x&=\int_{B_{R_2}}P(x,y)\mathcal A w(x)\,\d x-\int_{B_{R_2}\setminus \mathcal O}P(x,0)\mathcal A w(x)\,\d x\\
&=w(0)-\int_{\partial \mathcal O}P(x,0)\partial_{\nu_a}w(x)\,\d\sigma(x)+\int_{\partial \mathcal O}\partial_{\nu_a} P(x,0)w(x)\,\d\sigma(x),
\end{align*}
Combining this identity with \eqref{t3g} and applying the representation \eqref{t3f} of the solution in the transposition sense of \eqref{bvp1}, we deduce easily that \eqref{t3a} holds true.

We will know complete the proof of the theorem by showing that \eqref{t3b} implies \eqref{t3c}. For this purpose, let $v=v_1-v_2$ and notice that $v\in L^2(\mathcal O)$ is the solution in the transposition sense of
\begin{equation}\label{bvp2}\left\{\begin{array}{ll}\mathcal A v=\sum_{k=1}^{N_1}\lambda_k^1\delta_{x_k^1}-\sum_{k=1}^{N_2}\lambda_k^2\delta_{x_k^2},\quad &\textrm{in}\ \mathcal O,\\  v=0, &\textrm{on}\ \partial \mathcal O.\end{array}\right.\end{equation}
Choosing $r>0$ sufficiently small, we deduce that the set
$$U_r=\mathcal O\setminus\cup_{j=1}^2\cup_{k=1}^{N} \overline{B(x_k^1,r)}$$
is connected, $v|_{|U_r}\in H^2(U_r)$ and, applying \eqref{t3b}, we get
$$\left\{\begin{array}{ll}\mathcal A v=0,\quad &\textrm{in}\ U_r,\\  v=\partial_{\nu_a}v=0, &\textrm{on}\ S\subset\partial U_r.\end{array}\right.$$
 Therefore by unique continuation for elliptic equations we deduce that $v=0$ on $U_r$. Using the fact that here $r>0$ can be chosen arbitrary small and $v\in L^2(\mathcal O)$, we deduce that $v\equiv 0$ which implies that 
$$\sum_{k=1}^{N_1}\lambda_k^1\delta_{x_k^1}-\sum_{k=1}^{N_2}\lambda_k^2\delta_{x_k^2}=\mathcal A v\equiv0.$$
Combining this identity with the fact that $\lambda_k^j\neq0$, $j=1,2$, $k=1,\ldots,N_j$, we deduce that $N_1=N_2=N$ and there exists a one-to-one map $\sigma$ of $\{1,\ldots,N\}$ such that \eqref{t3c} holds true.
\end{proof}

\section*{Aknowledgement} The work of Y. Kian is supported by the French National Research Agency ANR and Hong Kong RGC Joint Research Scheme for the project IdiAnoDiff (grant ANR-24-CE40-7039). The work of H. Liu is supported by the Hong Kong RGC General Research Funds (projects 11311122, 11300821, and 11303125), the NSFC/RGC Joint Research Fund (project  N\_CityU101/21), the France-Hong Kong ANR/RGC Joint Research Grant, A-CityU203/19.

\end{document}